\documentclass[11pt]{article}%
\usepackage{graphicx}
\usepackage{amsmath}
\usepackage{amssymb}
\usepackage{float}
\usepackage{color}
\usepackage{multirow}
\usepackage{array}
\usepackage{epstopdf}
\usepackage{titlesec}
\usepackage{amsfonts}%
\providecommand{\U}[1]{\protect\rule{.1in}{.1in}}
\newcolumntype{P}[1]{>{\centering\arraybackslash}p{#1}}
\titleformat{\paragraph}
{\normalfont\normalsize\bfseries}{\theparagraph}{1em}{}
\titlespacing*{\paragraph}
{0pt}{3.25ex plus 1ex minus .2ex}{1.5ex plus .2ex}

\newtheorem {prop}{Proposition}[section]
\newtheorem {lemm}{Lemma}[section]

\newtheorem {rem}{Remark}[section]
\newcommand{\bull}{\vrule height7pt width7pt depth0pt}
\newenvironment{proof}
{\noindent {\it Proof :}}{\ \ \ \hfill\bull \vskip.5cm}

\newcommand{\iz}{{\mbox{\rm \rlap Z\kern 1.4pt Z}}}

\begin{document}

\title{Exact sampling for some multi-dimensional queueing models with renewal input}
\author{Jose Blanchet \thanks{Department of Management Science and Engineering,
Stanford University, Stanford, CA 94305.}
\and Yanan Pei \thanks{Department of Industrial Engineering and Operations
Research, Columbia University, New York, NY 10027.}
\and Karl Sigman \footnotemark[2]}
\maketitle

\begin{abstract}
Using a result of Blanchet and Wallwater (2015: Exact sampling of stationary
and time-reversed queues. \textit{ACM TOMACS}, \textbf{25}, 26) for exactly
simulating the maximum of a negative drift random walk queue endowed with
independent and identically distributed (iid) increments, we extend it to a
multi-dimensional setting and then we give a new algorithm for simulating
exactly the stationary distribution of a \textit{first-in-first-out} (FIFO)
multi-server queue in which the arrival process is a general renewal process
and the service times are iid; the FIFO $GI/GI/c$ queue with $2\leq c<\infty$.
Our method utilizes dominated coupling from the past (DCFP) as well as the
Random Assignment (RA) discipline, and complements the earlier work in which
Poisson arrivals were assumed, such as the recent work of Connor and Kendall
(2015: Perfect simulation of $M/G/c$ queues. \textit{Advances in Applied
Probability}, \textbf{47}, 4). We also consider the models in continuous-time
and show that with mild further assumptions, the exact simulation of those
stationary distributions can also be achieved. We also give, using our FIFO
algorithm, a new exact simulation algorithm for the stationary distribution of
the infinite server case, the $GI/GI/\infty$ model. Finally, we even show how
to handle \textit{Fork-Join} queues, in which each arriving customer brings
$c$ jobs, one for each server.


\end{abstract}

\baselineskip 12pt

\vskip .5 cm

\newpage

\section{Introduction}

In recent years, the method of \textit{exact simulation} has evolved as a
powerful way of sampling from stationary distributions of queueing models for
which such distributions can not be derived explicitly. The main method itself
is referred to as \textit{coupling from the past} (CFP) as introduced in Propp
and Wilson \cite{P-W} for finite-state discrete-time Markov chains. Since
then, the method has been generalized to cover general state space Markov
chains by using dominating processes; this is known as \textit{dominated
coupling from the past} (DCFP) as in Kendall \cite{Kendall}. The main purpose
of such methods is to produce a copy by simulation that exactly (not
approximately) has the stationary distribution desired. These methods involve
simulating processes backwards in time. In the present paper we consider using
such methods for the FIFO multi-server queue, denoted as the FIFO $GI/GI/c$
queue, $2\le c<\infty$, where $c$ denotes the number of servers working in
parallel, and arriving customers wait in one common queue (line).

The first algorithms yielding exact simulation in stationarity of the FIFO
$GI/GI/c$ queue are \cite{KS-I} and \cite{KS-II}, in which Poisson arrivals
are assumed; i.e., the $M/G/c$ case. In \cite{KS-I}, a DCFP method is used,
but the strong condition of \textit{super stability} is assumed, $\rho<1$,
instead of $\rho<c$. ($\rho=E(S)/E(T)$, where $T$ and $S$ denote an
interarrival time and service time respectively; stability only requires that
$\rho<c$.) As a dominating process, the $M/G/1$ queue is used under Processor
Sharing (PS) together (key) with its time-reversibility properties. In PS,
there is no line; all customers are served simultaneously but at a rate $1/n$
when there are $n\ge1$ customers in service. Then in \cite{KS-II}, the general
$\rho<c$ case is covered by using a forward time regenerative method (a
general method developed in \cite{A-G-T}) and using the $M/G/c$ model under a
random assignment (RA) discipline as an upper bound; a model in which each
arrival joins the $i^{th}$ queue with probability $1/c$ independently. (The
general forward-time regenerative method in \cite{A-G-T} unfortunately always
yields infinite expected termination time.) Then in \cite{C-K}, Connor and
Kendall generalize the DCFP/PS method in \cite{KS-I} by using the RA model.
They accomplish this by first exactly simulating the RA model in stationarity
backwards in time under PS at each node, then re-constructing it to obtain the
RA model with FIFO at each node and doing so in such a way that a sample-path
upper bound for the FIFO $M/G/c$ is achieved.

As for renewal arrivals (the general FIFO $GI/GI/c$ queue considered here) the
methods used above break down for various reasons, primarily because while
under Poisson arrivals the $c$ stations under RA become independent, they are
not independent for general renewal arrivals. Also, the time-reversibility
property of PS no longer holds, nor does Poisson Arrivals See Time Averages
(PASTA). Finally, under general renewal arrivals, the system may never empty
once it begins operating. New methods are needed. Blanchet, Dong and Pei
(\cite{B-D-P}) solve the problem by utilizing a vacation model as an upper
bound. In the present paper, however, we utilize DCFP by directly simulating
the RA model in reverse-time (under FIFO at each node). Our method involves
extending, to a multi-dimensional setting, a recent result of Blanchet and
Wallwater (\cite{B-W}) for exactly simulating the maximum of a negative drift
random walk endowed with iid increments. We also remark on how our approach
can lead to new results for other models too, such as multi-server queues
under the \textit{last-in-first-out} (LIFO) discipline, or the
\textit{randomly choose next} discipline, and even Fork-Join models (also
called split and match models).

\section{ The FIFO $GI/GI/c$ model}

\label{FIFO-model} Here we set up the classic first-in-first-out (of
queue/line) (FIFO) multi-server queueing model and its associated Markov chain
known as the \textit{Kiefer-Wolfowitz workload vector} (for further details,
see for example, page 341 in Chapter 12 of \cite{Asmussen-Book}, and the
original paper \cite{K-W-1955}) In what follows, as input to a $c$-server in
parallel multi-server queue with $c\ge2$, we have i.i.d. service times
$\{S_{n}:n\ge0\}$ distributed as $G(x) = P(S \le x),\ x\ge0$, with finite and
non-zero mean $0<E(S)=1/\mu<\infty$. Independently, the arrival times $\{t_{n}
: n\ge0\}$ ($t_{0}=0$) of customers to the model form a renewal process with
i.i.d. interarrival times $T_{n} = t_{n+1}-t_{n},\ n\ge0$ distributed as
$A(x)=P(T\le x),\ x\ge0,$ and finite non-zero arrival rate $0<\lambda=
{E(T)}^{-1}<\infty$. The FIFO $GI/GI/c$ model has only one queue (line), and
customers upon arrival join the end of the queue and then attend service at
the station that becomes free first (just like a USA postoffice). Because of
the FIFO assumption, the $n^{th}$ service time initiated for use by a server,
$S_{n}$, is used on the $n^{th}$ arrival (they arrive at time $t_{n}$), and so
one could equivalently imagine/assume that $S_{n}$ is brought to the system by
the $n^{th}$ arrival. In this equivalent form, we say that at time $t_{n}$,
the workload in the system has jumped upward by the amount $S_{n}$. Each
server is identical in that they process service times at rate $1$. We let
$\mathbf{ W}_{n}=\left(  W_{n}(1),\ldots, W_{n} (c)\right)  ^{T}$ denote the
Kiefer-Wolfowitz workload vector, defined recursively by
\begin{equation}
\label{K-W-discrete}\mathbf{ W}_{n+1}=\mathcal{R}\left(  \mathbf{ W}_{n} +
S_{n} \mathbf{e}-T_{n}\mathbf{f}\right)  ^{+},\ n\ge0,
\end{equation}
where $\mathbf{e}=(1,0,\ldots0)^{T}$, $\mathbf{f}=(1,1,\ldots, 1)^{T}$,
$\mathcal{R}$ places a vector in ascending order, and $^{+}$ takes the
positive part of each coordinate. The vector $\mathbf{W}_{n}$ shows, in
ascending order, how much work at time $t_{n}$ each server will process from
among all work in the system at that time not including $S_{n}$. Letting
$C_{n}$ denote the $n^{th}$ arriving customer, $D_{n}=W_{n}(1)$ is the
customer delay in queue (line) of $C_{n}$, because the server with the least
amount of work will be the first to empty in front of $C_{n}$. Recursion
(\ref{K-W-discrete}) defines a $c$-dimensional Markov chain due to the given
i.i.d. assumptions. The great importance of the recursion is that it yields
$\{D_{n}:n\ge0\}$, which is thus a function of a Markov chain.

With stability condition $\rho=\lambda/\mu<c$, it is well known that $\mathbf{
W}_{n}$ converges in distribution as $n\to\infty$ to a proper stationary
distribution (hence so does $D_{n}$). Let $\pi$ denote this stationary
distribution. Our main objective in the present paper is to provide a
simulation algorithm for sampling exactly from $\pi$.

\section{The RA $GI/GI/c$ model}

Given a $c$-server queueing system, the random assignment model (RA) is the
case when each of the $c$ servers forms its own FIFO single-server queue, and
each arrival to the system, independent of the past, randomly chooses queue
$i$ to join with equal probability $1/c,\ 1\le i\le c$. In the $GI/GI/c$ case,
we refer to this as the RA $GI/GI/c$ model. The following is a special case of
Lemma 1.3, Page 342 in \cite{Asmussen-Book}. (Such results and others even
more general are based on \cite{Wolff-bounds}, \cite{Foss}, and
\cite{Chernova-Foss}.)

\begin{lemm}
\label{l:main-inequality} Let $Q_{F}(t)$ denote total number of customers in
system at time $t\ge0$ for the FIFO $GI/GI/c$ model, and let $Q_{RA}(t)$
denote total number of customers in system at time $t\ge0$ for the
corresponding RA $GI/GI/c$ model in which both models are initially empty and
fed with exactly the same input of renewal arrivals $\{t_{n}:n\ge0\}$ and iid
service times $\{S_{n}:n\ge0\}$. Assume further that for both models the
service times are used by the servers in the order in which service
initiations occur ($S_{n}$ is the service time used for the $n^{th}$ such
initiation). Then
\begin{equation}
\label{e:main-inequality}P(Q_{F}(t)\le Q_{RA}(t),\ \hbox{for all}\ t\ge0)=1.
\end{equation}

\end{lemm}

The importance of Lemma~\ref{l:main-inequality} is that it allows us to
jointly simulate versions of the two stochastic processes $\{Q_{F}(t):t\ge0\}$
and $\{Q_{RA}(t):t\ge0\}$ while achieving a coupling such that
(\ref{e:main-inequality}) holds. In particular, whenever an arrival finds the
RA model empty, the FIFO model is found empty as well. (But we need to impose
further conditions if we wish to ensure that indeed the RA $GI/GI/c$ queue
will empty with certainty.) Letting time $t$ be sampled at arrival times of
customers, $\{t_{n} : n\ge0\}$, we thus also have%

\begin{equation}
\label{e:main-inequality-arrival-times}P(Q_{F}(t_{n}-)\le Q_{RA}%
(t_{n}-),\ \hbox{for all}\ n\ge0)=1.
\end{equation}

In other words, the total number in system as found by the $n^{th}$ arrival is
sample-path ordered as well. Note that for the FIFO model, the $n^{th}$
arriving customer $C_{n}$ initiates the $n^{th}$ service since FIFO means
``First-In-Queue-First-Out-of-Queue" where by ``queue" we mean the line before
entering service. This means that for the FIFO model we can attach $S_{n}$ to
$C_{n}$ upon arrival if we so wish when applying Lemma~\ref{l:main-inequality}%
. For the RA model, however, customers are not served in the order they
arrive. For example consider $c=2$ servers (system initially empty) and
suppose $C_{1}$ is assigned to node 1 with service time $S_{1}$, and $C_{2}$
also is assigned to node 1 (before $C_{1}$ departs) with service time $S_{2}$.
Meanwhile, before $C_{1}$ departs, suppose $C_{3}$ arrives and is assigned to
the empty node 2 with service time $S_{3}$. Then $S_{3}$ is used for the
second service initiation. \textit{For RA, the service times in order of
initiation are a random permutation of the originally assigned $\{S_{n}\}$.}

To use Lemma~\ref{l:main-inequality}, it is crucial to simply let the server
hand out service times one at a time when they are needed for a service
initiation. Thus, customers waiting in a queue before starting service do not
have a service time assigned until they enter service. In simulation
terminology, this amounts to generating the service times in order of when
they are needed.

One disadvantage of generating service times only when they are needed, is
that it then does not allow workload\footnote{Workload (total) at any time $t$
is defined as the sum of all whole and remaining service times in the system
at time $t$.} to be defined; only the amount of work in service. To get around
this if need be, one can simply generate service times upon arrival of
customers, and give them to the server to be used in order of service
initiation. The point is that when $C_{n}$ arrives, the total work in system
jumps up by the amount $S_{n}$. But $S_{n}$ is not assigned to $C_{n}$, it is
assigned (perhaps later) to which ever customer initiates the $n^{th}$
service. This allows Lemma~\ref{l:main-inequality} to hold true for total
amount of work in the system: If we let $\{V_{F}(t):t\ge0\}$ and
$\{V_{RA}(t):t\ge0\}$ denote total workload in the two models with the service
times used in the manner just explained, then in addition to
Lemma~\ref{l:main-inequality} we have
\begin{equation}
\label{e:main-inequality-work}P(V_{F}(t)\le V_{RA}(t),\ \hbox{for all}\ t\ge
0)=1,
\end{equation}
\begin{equation}
\label{e:main-inequality-work-arrival-times}P(V_{F}(t_{n}-)\le V_{RA}%
(t_{n}-),\ \hbox{for all}\ n\ge0)=1.
\end{equation}

It is important, however, to note that what one can't do is define workload at
the individual nodes $i$ by doing this, because that forces us to assign
$S_{n}$ to $C_{n}$ so that workload at the node that $C_{n}$ attends ($i$ say)
jumps by $S_{n}$ and $C_{n}$ enters service using $S_{n}$; that destroys the
proper coupling needed to obtain Lemma~\ref{l:main-inequality}. We can only
handle the total (sum over all $c$ nodes) workload. In the present paper, our
use of Lemma~\ref{l:main-inequality} is via a kind of reversal:

\begin{lemm}
\label{l: service-times-RA} Let $\{S_{n}^{\prime}\}$ be an iid sequence of
service times distributed as G, and assign $S_{n}^{\prime}$ to $C_{n}$ in the
RA model. Define $S_{n}$ as the service time used in the $n^{th}$ service
initiation. Then $\{S_{n}\}$ is also iid distributed as G.
\end{lemm}

\begin{proof}
The key is noting that we are re-ordering based only on the order in which
service times begin being used, not when they are completed (which would thus
introduce a bias). The service time chosen for the next initiation either
enters service immediately (e.g., is one that is routed to an empty queue by
an arriving customer) or is chosen from among those waiting in lines, and all
those waiting are iid distributed as $G$. Let ${\hat t}_{n}$ denote the time
at which the $n^{th}$ service initiation begins. The value $S_{n}$ of the
$n^{th}$ service time chosen (at time ${\hat t}_{n}$) by a server is
independent of the past service time values used before time ${\hat t}_{n}$,
and is distributed as $G$ (the choice of service time chosen as the next to be
used is not based on the value of the service time, only its position in the
lines). Letting $k(n)=$ the index of the $\{S^{\prime}_{n}\}$ that is chosen,
i.e., $S_{n}= S^{\prime}_{k(n)}$, it is this index (a random variable) that
depends on the past, but the value $S_{n}$ is independent of $k(n)$ since it
is a new one. Thus the $\{S_{n}\}$ are iid distributed as $G$.
\end{proof}

The point of the above Lemma~\ref{l: service-times-RA} is that we can, if we
so wish, simulate the RA model by assigning $S_{n}^{\prime}$ to $C_{n}$ (to be
used as their service time), but then assigning $S_{n}$, i.e. $S^{\prime
}_{k(n)}$, to $C_{n}$ in the FIFO model. By doing so the requirements of
Lemma~\ref{l:main-inequality} are satisfied and (\ref{e:main-inequality}),
(\ref{e:main-inequality-arrival-times}), (\ref{e:main-inequality-work}) and
(\ref{e:main-inequality-work-arrival-times}) hold. Interestingly, however, it
is not possible to first simulate the RA model up to a fixed time $t$, and
then stop and reconstruct the FIFO model up to this time $t$: At time $t$,
there may still be RA customers waiting in lines and hence not enough of the
$S_{n}$ have been determined yet to construct the FIFO model. But all we have
to do, if need be, is to continue the simulation of the RA model beyond $t$
until enough $S_{n}$ have been determined to construct fully the FIFO model up
to time $t$.

\section{Simulating exactly from the stationary distribution of the RA
$GI/GI/c$ model}

\label{sub:RA-simulation-exact}

By Lemma \ref{l:main-inequality}, the RA $GI/GI/c$ queue, which shares the
same arrival stream $\{t_{n}:n\ge0\}\ (t_{0}=0)$ and same service times in the
order of service initiations $\{S_{n}:n\ge0\}$, will serve as a sample path
upper bound (in terms of total number of customers in system and total
workload) of the target FIFO $GI/GI/c$ queue. Independent of $\{T_{n}:n\ge0\}$
and $\{S_{n}:n\ge0\}$, we let $\{U_{n}:n\ge0\}$ be an iid sequence of random
variables from discrete uniform distribution on $\{1,2,\ldots,c\}$; $U_{n}$
represents the choice that customer $C_{n}$ makes about which single-server
queue to join under RA discipline. Let $\mathbf{V}_{n}=\left(  V_{n}%
(1),\ldots,V_{n}(c)\right)  ^{T}$ denote the workload vector as found by
$C_{n}$ in the RA $GI/GI/c$ model, and for $i=1,\ldots,c$, $V_{n}(i)$ is the
waiting time of the $C_{n}$ if he chooses to join the FIFO single-server queue
of server $i$. So, $V_{0}(i)=0$ and
\begin{equation}
\label{e:RA-RRW-marginals}V_{n+1}(i)=\left(  V_{n}(i)+S_{n}I\left(
U_{n}=i\right)  -T_{n}\right)  ^{+},~~~n\ge0.
\end{equation}


These $c$ processes are dependent through the common arrival times
$\{t_{n}:n\ge0\}$ (equivalently common interarrival times $\{T_{n}:n\ge0\}$)
and the common $\{U_{n}:n\ge0\}$ random variables. Because of all the iid
assumptions, $\{\mathbf{V}_{n}:n\ge0\}$ forms a Markov chain. Define
$\tilde{\mathbf{S}}_{n} = \left(  S_{n}I\left(  U_{n}=1\right)  ,\ldots,
S_{n}I\left(  U_{n}=c\right)  \right)  ^{T}$ and $\mathbf{T}_{n} =
T_{n}\mathbf{f}$, then we can express (\ref{e:RA-RRW-marginals}) in vector
form as
\begin{equation}
\label{e:RA-RRW-vector}\mathbf{V}_{n+1} =\left(  \mathbf{V}_{n} +\tilde
{\mathbf{S}}_{n} -\mathbf{T}_{n}\right)  ^{+},~~~n\ge0.
\end{equation}
$\mathbf{V}_{n}$ uses the same interarrival times $\{T_{n}:n\ge0\}$ and
service times $\{S_{n}:n\ge0\}$ as we fed $\mathbf{W}_{n}$ in
(\ref{K-W-discrete}), however the coordinates of $\mathbf{V}_{n}$ are not in
ascending order, though all of them are nonnegative.

Each node $i$ as expressed in (\ref{e:RA-RRW-marginals}) can be viewed as a
FIFO $GI/GI/1$ queue with common renewal arrival process $\{t_{n}:n\ge0\}$,
but with iid service times $\{\tilde{S}_{n}(i)=S_{n}I(U_{n}=i):n\ge0\}$.
Across $i$, the service times $(\tilde{S}_{n}(1),\ldots,\tilde{S}_{n}(c))$ are
not independent, but they are identically distributed: marginally, with
probability $1/c$, $\tilde{S}_{n}(i)$ is distributed as $G$, and with
probability $(c-1)/c$ it is distributed as the point mass at $0$; i.e.,
$E(\tilde{S}(i)) = E(S)/c$. The point here is that we are not treating node
$i$ as a single-server queue endowed only with its own arrivals (a thinning of
the $\{t_{n}:n\ge0\}$ sequence) and its own service times iid distributed as
$G$. Defining iid increments $\Delta_{n}(i) = \tilde{S}_{n}(i)-T_{n}$ for
$n\ge0$, each node $i$ has an associated negative drift random walk
$\{R_{n}(i) : n\ge0\}$ with $R_{0}(i)=0$ and%

\begin{equation}
\label{e:RW-marginals}R_{n} (i)=\sum_{j=1}^{n} \Delta_{j} (i),~~~n\ge1.
\end{equation}

With $\rho= \lambda E(S)<c$, we define $\rho_{i}=\lambda E(\tilde
{S}(i))=\lambda E(S)/c=\rho/c < 1$; equivalently $E(\Delta(i)) < 0$ for all
$i=1,\ldots,c$. Let $V^{0}(i)$ denote a random variable with the limiting
(stationary) distribution of $V_{n}(i)$ as $n\to\infty$, it is well known (due
to the iid assumptions) that $V^{0}(i)$ has the same distribution as
\[
M(i)\triangleq\max_{m\ge0} R_{m}(i)
\]
for $i=1,\ldots,c$.

More generally, even when the increment sequence is just stationary ergodic,
not necessarily iid (hence not time reversible as in the iid case), it is the
backwards in time maximum that is used in constructing a stationary version of
$\{V_{n}(i)\}$. We will need this backwards approach in our simulation so we
go over it here; it is usually referred to as \textit{Loynes' Lemma}. We
extend the arrival point process $\{t_{n}:n\ge0\}$ to be a two-sided point
stationary renewal process $\{t_{n}:n\in\mathbb{Z}\}$
\[
\cdots t_{-2}<t_{-1}<0=t_{0}<t_{1}<t_{2}\cdots
\]
Equivalently, $T_{n}=t_{n+1}-t_{n},\ n\in\mathbb{Z},$ form iid interarrival
times; $\{T_{n}:n\in\mathbb{Z}\}$ forms a two-sided iid sequence.

Similarly, the iid sequences $\{S_{n}:n\ge0\}$ and $\{U_{n}:n\ge0\}$ are
extended to be two-sided iid, $\{S_{n}:n\in\mathbb{Z}\}$ and $\{U_{n}%
:n\in\mathbb{Z}\}.$ These extensions further allow two-sided extension of the
iid increment sequences $\{\Delta_{n}(i) : n\in\mathbb{Z}\}$ for
$i=1,\ldots,c$, i.e.,
\[
\Delta_{n}(i)=\tilde{S}_{n}-T_{n}=S_{n}I\left(  U_{n}=i\right)  -T_{n}%
,~~~n\in\mathbb{Z}.
\]


Then we define $c$ time-reversed (increments) random walks $\{R^{(r)}%
_{n}(i):n\ge0\}$ for $i=1,\ldots,c$, by $R_{0}^{(r)}(i)=0$ and
\begin{equation}
\label{e:RW-marginals-reverse-time}R_{n}^{(r)} (i)=\sum_{j=1}^{n} \Delta_{-j}
(i),~~~n\ge1.
\end{equation}

A (from the infinite past) stationary version of $\{V_{n}(i)\}$ denoted by
$\{V_{n}^{0}(i): n\le0\}$ is then constructed via
\begin{align}
V_{0}^{0}(i)  &  = \max_{m\ge0}R_{m}^{(r)} (i),\\
V_{-1}^{0}(i)  &  = \max_{m\ge1}R_{m}^{(r)} (i)- R_{1}^{(r)} (i),\\
V_{-2}^{0}(i)  &  = \max_{m\ge2}R_{m}^{(r)} (i)- R_{2}^{(r)} (i),\\
&  \vdots\nonumber\\
V_{-n}^{0}(i)  &  = \max_{m\ge n}R_{m}^{(r)} (i)- R_{n}^{(r)} (i),
\end{align}
for all $i=1,\ldots,c$.

By construction, the process $\mathbf{V}_{n}^{0} = (V_{n}^{0}(1),\ldots,
V_{n}^{0}(c))^{T},\ n\le0$, is jointly stationary representing a (from the
infinite past) stationary version of $\left\{  \mathbf{V}_{n}:n\le0\right\}
$, and satisfies the forward-time recursion (\ref{e:RA-RRW-vector}):
\begin{equation}
\label{e:from-past-recursion}\mathbf{V}_{n+1}^{0} =(\mathbf{V}_{n}^{0}
+\tilde{\mathbf{S}}_{n} -\mathbf{T}_{n})^{+},~~~n\le-1.
\end{equation}
Thus, by starting at $n =0$ and walking backwards in time, we have
(theoretically) a time-reversed copy of the RA model. Furthermore,
$\{\mathbf{V}_{n}^{0}:n\le0\}$ can be extended to include forward time $n\ge1$
via using the recursion further:
\begin{equation}
\label{e:into-future-recursion}\mathbf{V}_{n}^{0} =(\mathbf{V}_{n-1}^{0}
+\tilde{\mathbf{S}}_{n-1} -\mathbf{T}_{n-1})^{+},~~~n\ge1,
\end{equation}
where $\tilde{\mathbf{S}}_{n}=\left(  S_{n}I(U_{n}=1),\ldots,S_{n}%
I(U_{n}=c)\right)  ^{T}$ for $n\in\mathbb{Z}$.

In fact once we have a copy of just $\mathbf{V}_{0}^{0}$, we can start off the
Markov chain with it as initial condition and use
(\ref{e:into-future-recursion}) to obtain a forward in time stationary version
$\{\mathbf{V}_{n}^{0}:n\ge0\}$.

The above ``construction", however, is theoretical. We do not yet have any
explicit way of obtaining a copy of $\mathbf{V}_{0}^{0}$, let alone an entire
from-the-infinite-past sequence $\{ \mathbf{V}_{n}^{0}: n\le0 \}$. In Blanchet
and Wallwater \cite{B-W}, a simulation algorithm is given that yields (when
applied to each of our random walks), for each $1\le i\le c$, a copy of $\{(
R_{n}^{(r)} (i), V_{-n}^{0}(i) ) : 0\le n\le N\}$ for any desired $0\le
N<\infty$ including stopping times $N$. We modify the algorithm so that it can
do the simulation jointly across the $c$ systems, that is, we extend it to a
multi-dimensional form.

In particular, it yields an algorithm for obtaining a copy of $\mathbf{V}%
_{0}^{0}$, as well as a finite segment (of length $N$) of a backwards in time
copy of the RA model; $\{\mathbf{V}_{-n}^{0}:0\le n\le N\}$, a stationary into
the past construction up to discrete time $n=-N$.

Finite exponential moments are not required (because only \textit{truncated}
exponential moments are needed $E(e^{\gamma\Delta(i)}I\{|\Delta(i)|\le a\})$,
which in turn allow for the simulation of the exponential tilting of truncated
$\Delta(i)$, via acceptance-rejection). To get finite expected termination
time (at each individual node) one needs the service distribution to have
finite moment slightly beyond $2$: For some (explicitly known) $\epsilon>0$,
\begin{equation}
\label{big-ass}E(S^{2+\epsilon})<\infty.
\end{equation}

As our first case, we will be considering a stopping time $N$ such that
$\mathbf{V}_{-N}=\mathbf{0}$. Before we give the definition of the stopping
time $N$, we introduce the main idea of our simulation algorithm.

Let us define the maximum of a sequence of vectors. Suppose we have
$\mathbf{Z}_{1},\cdots,\mathbf{Z}_{k}$, where $\mathbf{Z}_{i}\in\mathbb{R}%
^{d}$ with $d\ge1$ and $k\in\mathbb{N}_{+}\cup\{\infty\}$, define
\[
\max\left(  \mathbf{Z}_{1},\cdots,\mathbf{Z}_{k}\right)  =\left(  \max_{1\le
i\le k}Z_{i}(1),\ldots,\max_{1\le i\le k}Z_{i}(d)\right)  ^{T}.
\]
Next define, for $n\in\mathbb{Z}$, that
\[
\mathbf{U}_{n}=\left(  I(U_{n}=1),\ldots,I(U_{n}=c)\right)  ^{T}
~~~~\mbox{and}~~~~ \mathbf{\Delta}_{n}=\tilde{\mathbf{S}}_{n}-\mathbf{T}%
_{n}=S_{n}\mathbf{U}_{n}-T_{n}\mathbf{f},
\]
where $\{U_{n}:n\in\mathbb{Z}\}$ are iid from discrete uniform distribution
over $\{1,2,\ldots,c\}$, and independently $\{T_{n}:n\in\mathbb{Z}\}$ are iid
from distribution $A$ (as introduced in Section \ref{FIFO-model}).
Our goal is to simulate the stopping time $N\in\mathbb{N}$ such that
$\mathbf{V}^{0}_{-N}=\mathbf{0}$, defined as
\begin{equation}
\label{eq-stopping-time}N=\inf\{n\ge0:\mathbf{V}^{0}_{-n}=\max_{k\ge
n}\mathbf{R}_{k}^{(r)}-\mathbf{R}_{n}^{(r)}=\mathbf{0}\},
\end{equation}
i.e. the first time walking in the past, that all coordinates of the workload
vector are $0$, jointly with $\{(\mathbf{R}_{n}^{(r)},\mathbf{V}^{0}%
_{-n}):0\le n\le N\}$. (By convention, the value of any empty sum of numbers
is zero, i.e. $\sum_{j=1}^{0}a_{j}=0$.)




To ensure that $E(N)<\infty$, in addition to $\rho<c$ (stability), it is
required that $P(T >S)>0$ (see the proof of Theorem 2 in \cite{KS-1988}), for
which the most common sufficient conditions are that $T$ has unbounded
support, $P(T > t) > 0,\ t \ge0$, or $S$ has mass arbitrarily close to $0$,
$P(S < t) > 0,\ t > 0$. But as we shall show in
Section~\ref{sub: bounded-interarrival-times}, given we know that $P(T >S)>0$,
we can assume without loss of generality that interarrival times are bounded.
It is that assumption which makes the extension of \cite{B-W} to a
multidimensional form easier to accomplish. Then, we show (in
Section~\ref{s:algorithm-sandwich} and Section~\ref{s:Harris-Chain}) how to
still simulate from $\pi$ even when $P(T>S)=0$. We do that in two different
ways, one as sandwiching argument and the other involving Harris recurrent
Markov chain regenerations.

\subsection{Algorithm for simulating exactly from $\pi$ for the FIFO $GI/GI/c$
queue: The case $P(T > S)>0$}

\label{sub:algorithm-main-sketch}

As mentioned earlier, we will assume that $P(T > S)>0$, so that the stable
($\rho< c$) RA and FIFO $GI/GI/c$ Markov chains (\ref{e:RA-RRW-vector}) and
(\ref{K-W-discrete}) will visit 0 infinitely often with certainty. (That the
RA model empties infinitely often when $P(T > S)>0$ is proved for example in
\cite{KS-1988}). We imagine that at the infinite past $n=-\infty$, we start
both (\ref{e:RA-RRW-vector}) and (\ref{K-W-discrete}) from empty. We construct
the RA model forward in time, while using Lemma~\ref{l: service-times-RA} for
the service times for the FIFO model, so that Lemma~\ref{l:main-inequality}
applies and we have it in the form of (\ref{e:main-inequality-arrival-times}),
for all $t_{n}\le0$ up to and including at time $t_{0}=0$, at which time both
models are in stationarity. We might have to continue the construction of the
RA model so that $\mathbf{W}_{0}$ (distributed as $\pi$) can be constructed
(e.g., enough service times have been initiated by the RA model for using
Lemmas~\ref{l:main-inequality} and \ref{l: service-times-RA}). Formally, one
can theoretically justify the existence of such infinite from the past
versions (that obey Lemma~\ref{l:main-inequality}) -- by use of Loynes' Lemma.
Each model (when started empty) satisfies the monotonicity required to use
Loynes' Lemma. In particular, noting that $Q_{RA}(t_{n}-) = 0$ if and only if
$\mathbf{V}_{n} = \mathbf{0}$, we conclude that if at any time $n$ it holds
that $\mathbf{V}_{n} = \mathbf{0}$, then $\mathbf{W}_{n} = \mathbf{0}$. By the
Markov property, given that $\mathbf{V}_{n} = \mathbf{0}=\mathbf{W}_{n}$, the
future is independent of the past for each model, or said differently,
\textit{the past is independent of the future}. This remains valid if $n$ is
replaced by a stopping time (strong Markov property).

We outline the simulation algorithm steps as follows.

\begin{enumerate}
\label{en:algorithm-main-sketch}

\item Simulate $\{\{( R_{n}^{(r)} (i), V_{-n}^{0}(i) ) : 0\le n \le N\}, 1\le
i\le c\}$ with $N$ as defined in (\ref{eq-stopping-time}). If $N=0$, go to
next step. Otherwise, having stored all data, reconstruct $\mathbf{V}_{n}^{0}$
forwards in time from $n=-N$ (initially empty) until $n=0$, using the
recursion (\ref{e:from-past-recursion}). During this forward-time
reconstruction, re-define $S_{j}$ as the $j$-th service initiation used by the
RA model (i.e., we are using Lemma~\ref{l: service-times-RA} to gather service
times in the proper order to feed in the FIFO model, which is why we do the
re-construction). If at time $n=0$ there have not yet been $N$ service
initiations, then continue simulating the RA model out in forward time until
finally there is a $N^{th}$ service initiation, and then stop. This will
require, at most, simulating out to $t_{n}$ with $n = N^{(+)} = \min\{n\ge0 :
\mathbf{V}_{n}^{0} =\mathbf{0}\}$.
Take the vector $(S_{-N},S_{-N+1},\ldots,S_{-1})$ and reset $(S_{0}%
,S_{1},\ldots,S_{N-1}) = (S_{-N},S_{-N+1},\ldots,S_{-1})$. Also, store the
interarrival times $(T_{-N},T_{-N+1},\ldots, T_{-1})$, and reset
$(T_{0},\ldots,T_{N-1}) = (T_{-N},T_{-N+1},\ldots, T_{-1})$.

\item If $N=0$, then set $\mathbf{W}_{0}=\mathbf{0}$ and stop. Otherwise use
(\ref{K-W-discrete}) with $\mathbf{W}_{0}=\mathbf{0}$, recursively go forwards
in time for $N$ steps until obtaining $\mathbf{W}_{N}$, by using the $N$
re-set service $(S_{0},S_{1},\ldots,S_{N-1})$ and interarrival times
$(T_{0},\ldots,T_{N-1})$. Reset $\mathbf{W}_{0}=\mathbf{W}_{N}$.

\item Output $\mathbf{W}_{0}$.
\end{enumerate}

Detailed simulation steps are discussed in Appendix (Section
\ref{appendix-simulation-algo}). Let $\tau$ denote the total number of
interarrival times and service times to simulate in order to detect the
stopping time $N$. The following proposition shows that our algorithm will
terminate in finite expected time, i.e., $\mathbb{E}(\tau)<\infty$. The proof
is given in Section \ref{proofs}.

\begin{prop}
\label{alg1-prop} If $\rho=\lambda/\mu<c$, $P(T>S)>0$, and there exists some
$\epsilon>0$ such that $E(S^{2+\epsilon})<\infty$, then
\[
E(N)<\infty~~~~\mbox{and}~~~~E(\tau)<\infty.
\]

\end{prop}

\subsection{A more efficient algorithm: sandwiching}

\label{s:algorithm-sandwich}

In this section, we no longer even need to assume that $P(T>S)>0$. (Another
method allowing for $P(T>S)=0$ involving Harris recurrent regeneration is
given later in Section~\ref{s:Harris-Chain}.) Instead of waiting for the
workload vector of the $GI/GI/c$ queue under RA discipline to become
$\mathbf{0}$
, we choose an ``inspection time" $t_{-\kappa}<0$ for some $\kappa
\in\mathbb{Z}_{+}$ to stop the backward simulation of the RA $GI/GI/c$ queue,
then construct two bounding processes of the target FIFO $GI/GI/c$ queue and
evolve them forward in time, using the same stream of arrivals and service
time requirements (in the order of service initiations), until coalescence or
time zero. In particular we let the upper bound process to be a FIFO $GI/GI/c$
queue starting at time $t_{-\kappa}$ with workload vector being $\mathbf{V}%
^{0}_{-\kappa}$, and let the lower bound process to be a FIFO $GI/GI/c$ queue
starting at the same time $t_{-\kappa}$ from empty, i.e., with workload vector
being $\mathbf{0}$.




Let $\mathbf{W}(t)$
denote the ordered (ascendingly) workload vector of the original FIFO
$GI/GI/c$ queueing process, starting from the infinite past, evaluated at time
$t$. For $t\ge t_{-\kappa}$, we define $\mathbf{W}^{u}_{-\kappa}(t)$
and $\mathbf{W}^{l}_{-\kappa}(t)$
to be the ordered (ascendingly) workload vectors of the upper bound and lower
bound processes, initiated at the inspection time $t_{-\kappa}$, evaluated at
time $t$. By our construction and Theorem 3.3 in \cite{C-K},
\[
\mathbf{W}_{-\kappa}^{u}(t_{-\kappa})=\mathcal{R}\left(  \mathbf{V}%
^{0}_{-\kappa}\right)  \ge\mathbf{W}(t_{-\kappa})\ge\mathbf{W}_{-\kappa}%
^{l}(t_{-\kappa})=\mathbf{0},
\]
and for all $t>t_{-\kappa}$
\[
\mathbf{W}^{u}_{-\kappa}(t)\ge\mathbf{W}(t)\ge\mathbf{W}^{l}_{-\kappa}(t),
\]
where all the above inequalities hold coordinate-wise.

Note that we can evolve the ordered workload vectors of the two bounding
processes as follows: for $t_{n-1}\le t< t_{n}$ when $-\kappa<n\le-1$,
\begin{align}
\label{e:workload-not-at-arrival}%
\begin{split}
\mathbf{W}_{-\kappa}^{u}(t)  &  =\mathcal{R}\left(  \mathbf{W}^{u}_{-\kappa
}(t_{n-1})+S_{n-1}\mathbf{e}-(t-t_{n-1})\mathbf{f}\right)  ^{+},\\
\mathbf{W}_{-\kappa}^{l}(t)  &  =\mathcal{R}\left(  \mathbf{W}^{l}_{-\kappa
}(t_{n-1})+S_{n-1}\mathbf{e}-(t-t_{n-1})\mathbf{f}\right)  ^{+}.
\end{split}
\end{align}

Similarly let $Q(t)$ denote the number of customers in the original FIFO
$GI/GI/c$ queueing process, starting from the infinite past, evaluated at time
$t$. For $t\ge t_{-\kappa}$, we let $Q^{u}_{-\kappa}(t)$ and $Q^{l}_{-\kappa
}(t)$ denote the number of customers in the upper and lower bound queueing
processes respectively, both initiated at the inspection time $t_{-\kappa}$,
evaluated at time $t$. If at some time $\tau\in[t_{-\kappa},0]$, we observe
that $\mathbf{W}^{u}_{-\kappa}(\tau)=\mathbf{W}^{l}_{-\kappa}(\tau)$, then it
must be true that $\mathbf{W}(\tau)=\mathbf{W}^{u}_{-\kappa}(\tau
)=\mathbf{W}^{l}_{-\kappa}(\tau)$ and $Q(\tau)=Q_{-\kappa}^{u}(\tau
)=Q_{-\kappa}^{l}(\tau)$ (because the ordered remaining workload vectors of
two bounding processes can only meet when they both have idle servers). We
call such time $\tau$ ``coalescence time" and from then on we have full
information of the target FIFO $GI/GI/c$ queue, hence we can continue simulate
it forward in time until time $0$.

However if coalescence does not happen by time $0$, we can adopt the so-called
``binary back-off" method by letting the arrival time $t_{-2\kappa}$ be our
new inspection time and redo the above procedure to detect coalescence.
Theorem 3.3 in \cite{C-K} ensures that for any $t_{-\kappa}\le t\le0$
\[
\mathbf{W}_{-\kappa}^{u}(t)\ge\mathbf{W}_{-2\kappa}^{u}(t)\ge\mathbf{W}%
(t)\ge\mathbf{W}_{-2\kappa}^{l}(t)\ge\mathbf{W}_{-\kappa}^{l}(t).
\]

We summarize the sandwiching algorithm as follows.

\begin{enumerate}
\item Simulate $\{(\mathbf{R}_{n}^{(r)},\mathbf{V}_{-n}^{0}):0\le n\le
\kappa\}$ with all data stored.

\item Use the stored data to reconstruct $\mathbf{V}_{n}^{0}$ forward in time
from $n=-\kappa$ until $n=0$, using Equation (\ref{e:from-past-recursion}),
and re-define $S_{j}$ as the $j^{th}$ service initiation used by the RA model.

\item Set $\mathbf{W}_{-\kappa}^{u}(t_{-\kappa})=\mathcal{R}(\mathbf{V}%
_{-\kappa}^{0})$ and $\mathbf{W}_{-\kappa}^{l}(t_{-\kappa})=\mathbf{0}$. Then
use the same stream of interarrival times $\left(  T_{-\kappa},T_{-\kappa
+1},\cdots,T_{-1}\right)  $ and service times $\left(  S_{-\kappa}%
,S_{-\kappa+1},\cdots,S_{-1}\right)  $ to simulate $\mathbf{W}^{u}_{-\kappa
}(t)$, $\mathbf{W}^{l}_{-\kappa}(t)$ forward in time using Equation
(\ref{e:workload-not-at-arrival}).

\item If at some time $t\in[t_{-\kappa},0]$ we detect $\mathbf{W}^{u}%
_{-\kappa}(t)=\mathbf{W}^{l}_{-\kappa}(t)$, set $\tau=t$, $\mathbf{W}%
(\tau)=\mathbf{W}^{u}_{-\kappa}(\tau)$, $Q(\tau)=\sum_{i=1}^{c}I(\mathbf{W}%
(\tau;i)>0)$, where $\mathbf{W}(t;i)$ is the $i$-th entry of vector
$\mathbf{W}(t)$. Then use the remaining interarrival times and service times
to evolve the original FIFO $GI/GI/c$ queue forward in time until time
$t_{0}=0$, output $(\mathbf{W}(0),Q(0))$ and stop.

\item If no coalescence is detected by time $0$, set $\kappa=2\kappa$, then
continue to simulate the backward RA $GI/GI/c$ process until $(-\kappa)$-th
arrival, i.e., $\{(\mathbf{R}_{n}^{(r)},\mathbf{V}_{-n}^{0}):0\le n\le
\kappa\}$, with all data stored. Go to Step 2.
\end{enumerate}

Next we analyze properties of the coalescence time. Define
\[
\kappa^{*}_{-}=\inf\left\{  n\ge0:\inf_{t_{-n}\le t\le0}\lVert\mathbf{W}%
^{u}_{-n}(t)-\mathbf{W}^{l}_{-n}(t)\rVert_{\infty}=0\right\}  .
\]
If at time $t_{-\kappa^{*}_{-}}$ we start an upper bound FIFO $GI/GI/c$ queue
with workload vector being $\mathbf{W}^{u}_{-\kappa_{-}^{*}}(t_{-\kappa
_{-}^{*}})$, and a lower bound FIFO $GI/GI/c$ queue with workload vector being
$\mathbf{0}$, they will coalesce by time $t_{0}=0$. Therefore if we simulate
the RA system backwards in time to $t_{-\kappa_{-}^{*}}$, we will be able to
detect a coalescence. We next show that $E(-t_{-\kappa_{-}^{*}})<\infty$.

By stationarity we have that $\kappa^{*}_{-}$ is equal in distribution to
\[
\kappa^{*}_{+}=\inf\left\{  n\ge0:\inf_{0\le t\le t_{n}}\lVert\mathbf{W}%
^{u}_{0}(t)-\mathbf{W}^{l}_{0}(t)\rVert_{\infty}=0\right\}  ,
\]
hence $-t_{-\kappa^{*}_{-}}\overset{d}{=}t_{\kappa^{*}_{+}}$.

\begin{prop}
\label{alg2-prop} If $\rho=E(S)/E(T)<c$ and there exists some $\epsilon>0$
such that $E(S^{2+\epsilon})<\infty$ and $E(T^{2+\epsilon})<\infty$, then
\[
E(t_{\kappa^{*}_{+}})<\infty.
\]

\end{prop}

The proof follows the same argument as in the proof of Proposition 3 in
\cite{B-D-P}, so we give a brief proof outline in Section \ref{proofs}.

\section{Continuous-time stationary constructions}

For a stable FIFO $GI/GI/1$ queue, let $D$ denote stationary customer delay
(time spent in queue (line)); i.e., it has the limiting distribution of
$D_{n+1} = (D_{n} + S_{n}-T_{n})^{+}$ as $n\to\infty$.

Independently, let $S_{e}$ denote a random variable distributed as the
\textit{equilibrium distribution} $G_{e}$ of service time distribution $G$,
\begin{equation}
G_{e}(x)=\mu\int_{0}^{x} P(S>y)dy,\ x\ge0,
\end{equation}
where $S\sim G$. Let $V(t)$ denote total work in system at time $t$; the sum
of all whole or remaining service times in the system at time $t$. $D_{n} = V
(t_{n}-)$, and one can construct $\{V(t)\}$ via
\[
V(t)=(D_{n} +S_{n}-(t-t_{n}))^{+},\ t_{n}\le t < t_{n+1}.
\]
(It is to be continuous from the right with left limits.) Let $V$ denote
stationary workload; i.e., it has the limiting distribution
\begin{equation}
P(V\le x)=\lim_{t\to\infty} {\frac{1}{t}} \int_{0}^{t} P(V(s)\le x)ds,\ x\ge0.
\end{equation}

The following is well known to hold (see Section 6.3 and 6.4 in \cite{KS-Book}%
, for example):
\begin{equation}
P(V >x)=\rho P(D+S_{e} >x),\ x\ge0.
\end{equation}

Letting $F_{D}(x)=P(D\le x)$ denote the probability distribution of $D$, and
$\delta_{0}$ denote the point mass at $0$, and $\ast$ denote convolution of
distributions, this means that the distribution of $V$ can be written as a
mixture
\[
(1-\rho)\delta_{0} +\rho F_{D}\ast G_{e}.
\]

This leads to the following:

\begin{prop}
For a stable $(0<\rho<1)$ FIFO $GI/GI/1$ queue, if $\rho$ is explicitly known,
and one can exactly simulate from $D$ and $G_{e}$, then one can exactly
simulate from $V$.
\end{prop}

\begin{proof}
\begin{enumerate}
\item Simulate a Bernoulli $(\rho)$ r.v. B.

\item If $B=0$, then set $V=0$. Otherwise, if $B=1$, then: simulate $D$ and
independently simulate a copy $S_{e}\sim G_{e}$. Set $V = D + S_{e}$. Stop.
\end{enumerate}
\end{proof}

Another algorithm requiring instead the ability to simulate from $A_{e}$
(equilibrium distribution of the interarrival-time distribution $A$) instead
of $G_{e}$ follows from another known relation:%

\begin{equation}
\label{e:hooks-law-1}V\overset{d}{=} (D+S-T_{e})^{+},
\end{equation}
where $D, S$ and $T_{e}\sim A_{e}$ are independent (see, for example Equation
(88) on Page 426 in \cite{Wolff-Book}). Thus by simulating $D, S$, and $T_{e}%
$, simply set $V=(D+S-T_{e})^{+}$. Equation (\ref{e:hooks-law-1}) extends
analogously to the FIFO $GI/GI/c$ model, where our objective is to exactly
simulate from the time-stationary distribution of the continuous-time
Kiefer-Wolfowitz workload vector, $\mathbf{W}(t) = (W(t;1),\ldots
,W(t;c))^{T},\ t\ge0$, where it can be constructed via
\[
\mathbf{W}(t)=\mathcal{R}(\mathbf{W}_{n} +S_{n} \mathbf{e}-(t-t_{n}%
)\mathbf{f})^{+},\ t_{n}\le t<t_{n+1}.
\]
It is to be continuous from the right with left limits; $\mathbf{W}_{n} =
\mathbf{W}(t_{n}-)$. Total workload $V(t)$, for example, is obtained from this
via
\[
V(t)=\sum_{i=1}^{c} W(t;i).
\]

Letting $\mathbf{W}^{*}$ have the time-stationary distribution of
$\mathbf{W}(t)$ as $t\to\infty$, and letting $\mathbf{W}_{0}$ have the
discrete-time stationary distribution $\pi$ and letting $S$, $T_{e}$ and
$\mathbf{W}_{0}$ be independent, then
\begin{equation}
\label{e:time-versus-arrival}\mathbf{W}^{*}\overset{d}{=} \mathcal{R}%
(\mathbf{W}_{0}+S\mathbf{e}-T_{e}\mathbf{f})^{+}.
\end{equation}
So once we have a copy of $\mathbf{W}_{0}$ (distributed as $\pi$) from our
algorithm in Section~\ref{sub:algorithm-main-sketch} or
Section~\ref{s:algorithm-sandwich}, we can easily construct a copy of
$\mathbf{W}^{*}$ as long as we can simulate from $A_{e}$. Of course, if
arrivals are Poisson then the distribution of $\mathbf{W}^{*}$ is identical to
that of $\mathbf{W}_{0}$ by PASTA, but otherwise we can use
(\ref{e:time-versus-arrival}).

\subsection{Numerical Results}


As a sanity check, we have implemented our perfect sampling algorithm in
Matlab for the case of $Erlang(k_{1},\lambda)/Erlang(k_{2},\mu)/c$ queue. We
provide our implementation codes for both algorithms in the online appendix of
this paper.

Firstly we consider $M/M/c$ queues, which are special cases of $Erlang(k_{1}%
,\lambda)/Erlang(k_{2},\mu)/c$ with $k_{1}=k_{2}=1$. For the quantity of
interest, number of customers in the FIFO $M/M/c$ queue at stationary, we
obtain its empirical distribution from a large number of independent runs of
our algorithm and compare it to the theoretical distribution which has a
well-established closed form
\begin{align*}
\pi_{0}  &  =\left(  \sum_{k=0}^{c-1}\frac{\rho^{k}}{k!}+\frac{\rho^{c}%
}{(c-1)!}\frac{1}{c-\rho}\right)  ^{-1},\\
\pi_{k}  &  =%
\begin{cases}
\pi_{0}\cdot\rho^{k}/k! & \text{if $0<k<c$}\\
\pi_{0}\cdot\rho^{k}c^{c-k}/c! & \text{if $k\ge c$}%
\end{cases}
,
\end{align*}
where $\rho=\lambda/\mu<c$.

As an example, Figure \ref{fig:valid-check-1} shows the result of such test
when $\lambda=3$, $\mu=2$ and $c=2$. Grey bars are the empirical results of
$5,000$ draws using our algorithm, and black bars are the theoretical
distribution number of customers in system from stationarity. A Pearson's
chi-squared test between the theoretical and empirical distributions gives a
$p$-value equal to $0.8781$, indicating close agreement (i.e., we cannot
reject the null-hypothesis that there is no difference between these two
distributions). For another set of parameters $\lambda=10$, $\mu=2$ and
$c=10$, the results are shown in Figure \ref{fig:valid-check-2} with a
$p$-value being $0.6069$ for the chi-squared fitness test.

For the general $Erlang(k_{1},\lambda)/Erlang(k_{2},\mu)/c$ queue when
$k_{1}>1$ and $k_{2}>1$ when $\rho/c=\lambda k_{2}/(c\mu k_{1})=0.9$, we
compare the empirical distribution of number of customers in system at
stationarity, obtained from a large number of runs of our perfect sampling
algorithm, to the numerical results (with precision at least $10^{-4}$)
provided in Table III of \cite{Hillier-Lo}. The results for an
$Erlang(2,9)/Erlang(2,5)/c$ queue are given in Figure
\ref{fig:valid-check-erlang}. Grey bars are the empirical results of $5,000$
draws using our algorithm and black bars are the numerical values given in
\cite{Hillier-Lo}, and they are very close to each other. The Pearson's
chi-squared test gives a $p$-value of $0.9464$, thus we cannot reject the
null-hypothesis that these two distributions agree well.


\begin{figure}[h]
\centering
\includegraphics[width=10cm]{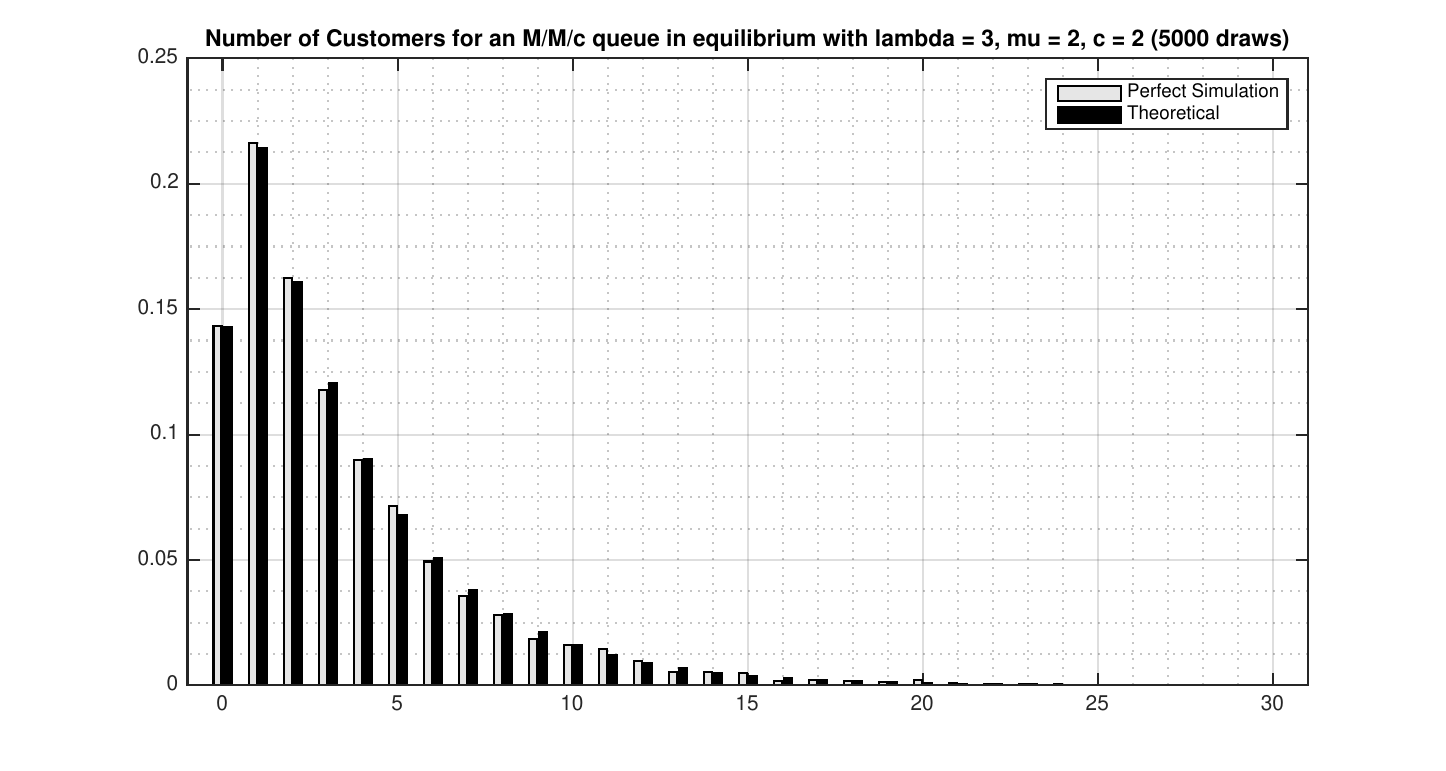}
\caption{Number of customers for an $M/M/c$ queue in stationarity when
$\lambda=3$, $\mu=2$, $c=2$. }%
\label{fig:valid-check-1}%
\end{figure}\begin{figure}[h]
\centering
\includegraphics[width=10cm]{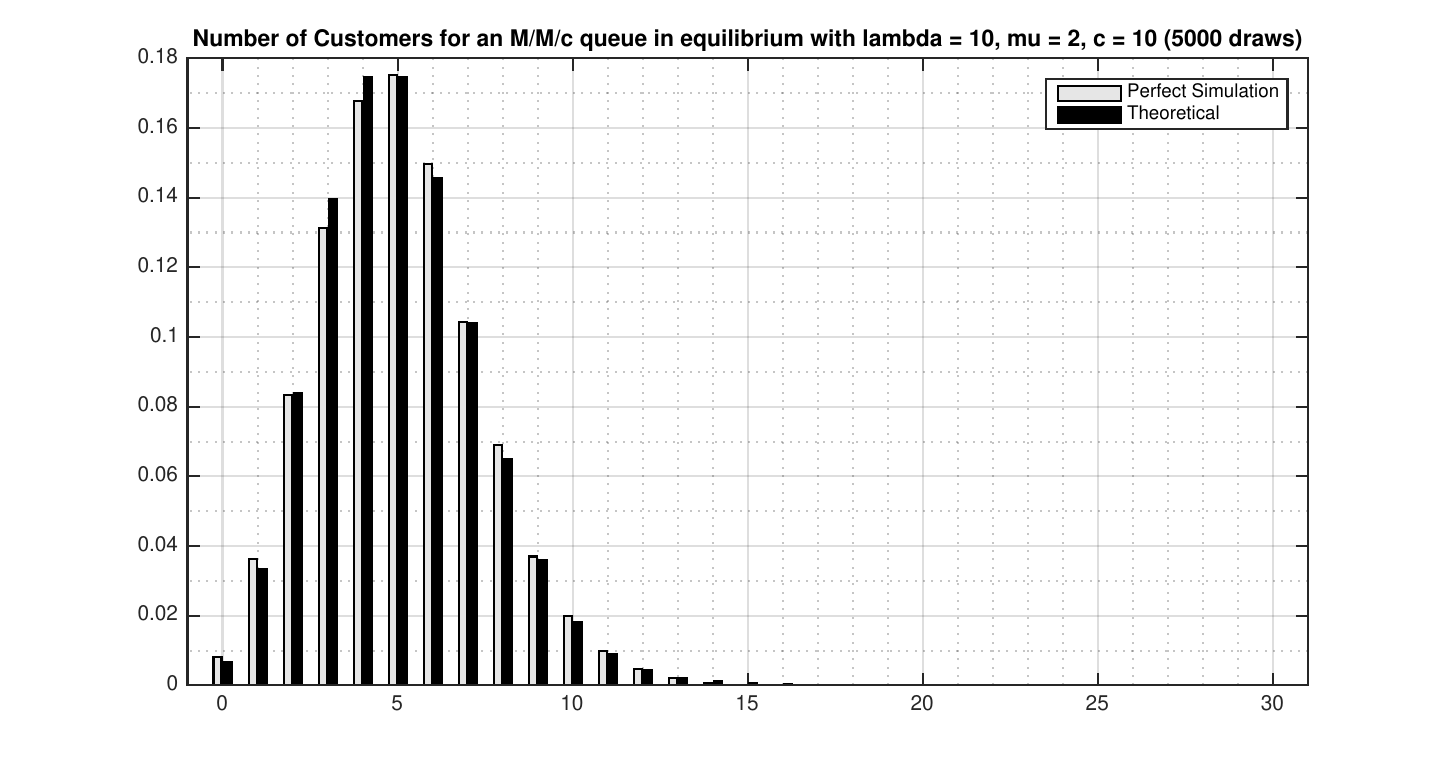}
\caption{Number of customers for an $M/M/c$ queue in stationarity when
$\lambda=10$, $\mu=2$, $c=10$.}%
\label{fig:valid-check-2}%
\end{figure}\begin{figure}[h]
\centering
\includegraphics[width=10cm]{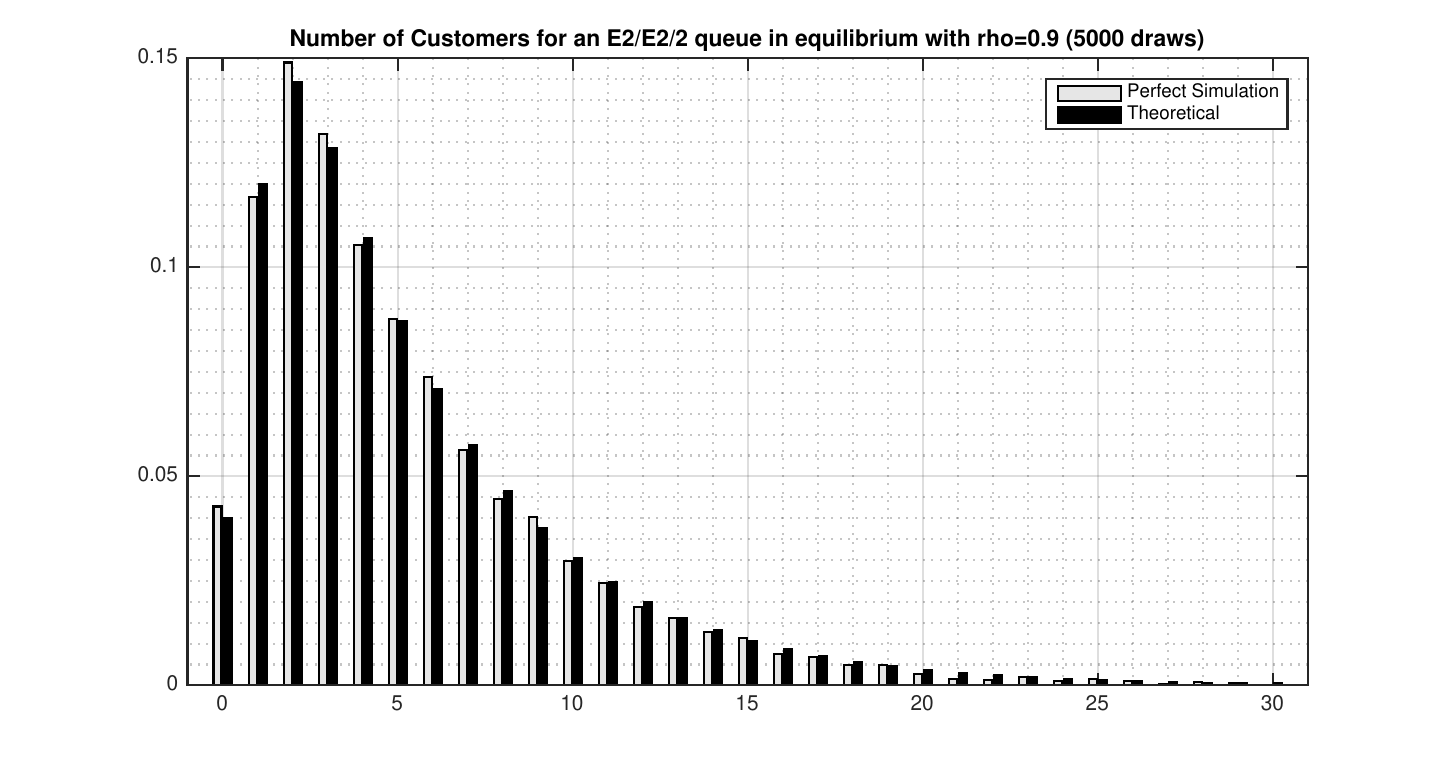}
\caption{Number of customers for an $Erlang(k_{1},\lambda)/Erlang(k_{2}%
,\mu)/c$ queue in stationarity when $k_{1}=2$, $\lambda=9$, $k_{2}=2$, $\mu
=5$, $c=2$ and $\rho/c=0.9$.}%
\label{fig:valid-check-erlang}%
\end{figure}

Next we run a numerical experiment comparing the computational efficiency of
the first algorithm in Section \ref{sub:algorithm-main-sketch} and the second
sandwiching algorithm in Section \ref{s:algorithm-sandwich}. We measure the
computational efficiency in two aspects. The first one is how far in the past
we need to simulate the dominating process to detect coalescence (counting the
total number of arrivals sampled backwards). The second aspect involves actual
computation time in seconds. Figure \ref{fig:2-alg-compare} depicts such a
comparison for an $M/M/c$ queue with parameters $\lambda=10$, $\mu=2$, $c=10$,
from $5000$ runs. Both results indicate that the second algorithm
(sandwiching) is significantly more efficient than the first one.

\begin{figure}[ptb]
\centering
\includegraphics[width=14cm]{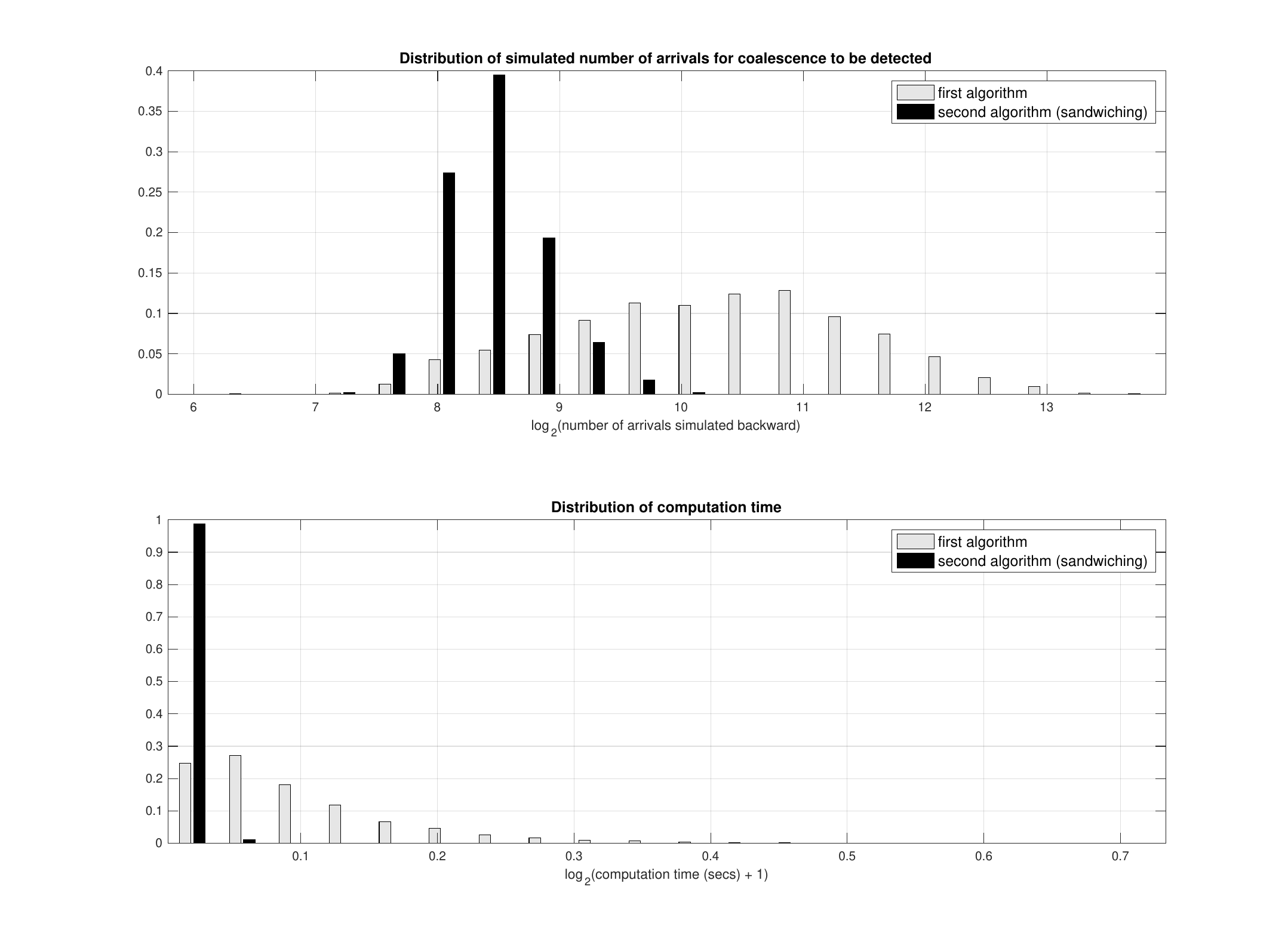}\caption{Computational
efficiency comparison between two algorithms for an $M/M/c$ queue}%
\label{fig:2-alg-compare}%
\end{figure}

Finally we study how the computational complexity of our sandwiching algorithm
compares to the algorithm given in \cite{B-D-P}. Notice these two algorithms
look similar: they both use back-off strategies to run two bounding processes
from some inspection time and check if they meet before time $0$. The
difference is that in \cite{B-D-P} they use a so-called \textquotedblleft
vacation system\textquotedblright\ to construct upper bound process, whereas
we use the same queue but under RA discipline instead. In the following
numerical experiment, we define the computational complexity as the total
number of arrivals each algorithm samples backwards to detect coalescence.
Table \ref{tab:computation-cmpx-test} shows how they vary with traffic
intensity, $\rho/c=\lambda/(c\mu)$, based on $5000$ independent runs of both
algorithms using the same back-off strategy with same initial $\kappa=1$. The
result suggests that our second algorithm (sandwiching) outperforms the one
proposed in \cite{B-D-P} as the magnitude of our computational complexity does
not increase as fast as theirs when traffic intensity increases.

\begin{table}[h]
\caption{simulation result for computational complexities with varying traffic
intensities}%
\label{tab:computation-cmpx-test}%
\centering
{\small $M/M/c$ queue with fixed $\mu=5$ and $c=2$} \bigskip
\par
\renewcommand{\arraystretch}{1.5}
\begin{tabular}{ | P{1cm} | P{1cm}| P{16em}| P{14em}|} \hline
\multirow{2}{.5cm}{$\lambda$} & \multirow{2}{.5cm}{$\rho/c$} &  \multicolumn{2}{|c|}{$95\%$ confidence interval of number of arrivals simulated backwards}\\
\cline{3-4}
& & Algorithm in Section \ref{s:algorithm-sandwich} & Algorithm in \cite{B-D-P}\\
\hline
5 & 0.5 & 54.8194 $\pm$ 0.5758 & 146.5618 $\pm$ 2.3598 \\
\hline
6 & 0.6 & 86.5394 $\pm$ 1.0536 & 308.4448 $\pm$ 4.9413  \\
\hline
7 & 0.7 & 152.6552 $\pm$ 2.2695 & 730.1130 $\pm$ 11.2783 \\
\hline
8 & 0.8 & 337.9544 $\pm$ 6.3021 & 2201.8254 $\pm$ 32.1556 \\
\hline
9 & 0.9 & 1521.3502 $\pm$ 31.8267	 & 12277.8686 $\pm$ 161.5824\\
\hline
\end{tabular}
\end{table}

\section{Why we can assume that interarrival times are bounded}

\label{sub: bounded-interarrival-times}

\begin{lemm}
\label{l:upper-bound-RRW} Consider the recursion
\begin{equation}
\label{e:RRW}D_{n+1}=(D_{n}+S_{n}-T_{n})^{+},\ n\ge0,
\end{equation}
where both $\{T_{n}\}$ and $\{S_{n}\}$ are non-negative random variables, and
$D_{0}=0$.

Suppose for another sequence of non-negative random variables $\{{\hat T}%
_{n}\}$, it holds that
\[
P({\hat T}_{n}\le T_{n},\ n\ge0)=1.
\]
Then for the recursion
\begin{equation}
\label{e:RRW-bounding}{\hat D}_{n+1}=({\hat D}_{n}+S_{n}-{\hat T}_{n}%
)^{+},\ n\ge0,
\end{equation}
with ${\hat D}_{0}=0$, it holds that
\begin{equation}
\label{e:RRW-bound}P(D_{n}\le{\hat D}_{n},\ n\ge0)=1.
\end{equation}

\end{lemm}

\begin{proof}
The proof is by induction on $n\ge0$: Because (w.p.1 in the following
arguments) ${\hat T}_{0}\le T_{0}$, we have
\[
D_{1}=(S_{0}-T_{0})^{+}\le(S_{0}-{\hat T}_{0})^{+}={\hat D}_{1}.
\]
Now suppose the result holds for some $n\ge0$. Then $D_{n}\le{\hat D}_{n}$ and
by assumption ${\hat T}_{n}\le T_{n}$; hence
\[
D_{n+1}=(D_{n}+S_{n}-T_{n})^{+}\le({\hat D}_{n}+S_{n}-{\hat T}_{n})^{+}={\hat
D}_{n+1},
\]
and the proof is complete.
\end{proof}

\begin{prop}
Consider the stable RA $GI/GI/c$ model in which $P(T>S)>0$. In order to use
this model to simulate from the corresponding stationary distribution of the
FIFO $GI/GI/c$ model as explained in the
Section~\ref{sub:algorithm-main-sketch}, without loss of generality we can
assume that the interarrival times $\{T_{n}\}$ are bounded: There exists $b>0$
such that
\[
P(T_{n}\le b,\ n\ge0)=1.
\]

\end{prop}

\begin{proof}
By stability, $cE(T)>E(S)$, and by assumption $P(T>S)>0$. If the $\{T_{n}\}$
are not bounded, then for $b>0$, define ${\hat T}_{n}=\min\{T_{n},
b\},\ n\ge0$; truncated $T_{n}$. Choose $b$ sufficiently large so that
$cE({\hat T})>E(S)$ and $P({\hat T}>S)>0$ still holds. Now use the $\{{\hat
T}_{n}\}$ in place of the $\{T_{n}\}$ to construct an RA model, denoted by
$\widehat{RA}$. Denote this by
\[
\mathbf{{\hat V}}_{n} = ({\hat V}_{n}(1),\ldots, {\hat V}_{n} (c)),
\]
where it satisfies the recursion (\ref{e:RA-RRW-vector}) in the form
\[
\mathbf{{\hat V}}_{n+1} = (\mathbf{{\hat V}}_{n} + \mathbf{\tilde{S}}%
_{n}-\mathbf{{\hat T}}_{n})^{+},\ n\ge0,
\]
where $\mathbf{\hat{T}}_{n}=\hat{T}_{n}\cdot\mathbf{f}$.

Starting from $\mathbf{V}_{0}=\mathbf{\hat{V}}_{0}=\mathbf{0}$, then from
Lemma~\ref{l:upper-bound-RRW}, it holds (coordinate-wise) that
\[
\mathbf{V}_{n}\le\mathbf{{\hat V}}_{n},\ n\ge0,
\]
and thus, if for some $n\ge0$ it holds that $\mathbf{{\hat V}}_{n}=\mathbf{0}%
$, then $\mathbf{V}_{n}=\mathbf{0}$ and hence $\mathbf{W}_{n}=\mathbf{0}$ (as
explained in our previous section). Since $b$ was chosen ensuring that
$cE({\hat T})>E(S)$ and $P({\hat T}>S)>0$, $\{\mathbf{{\hat V}}_{n}\}$ is a
stable RA $GI/GI/c$ queue that will indeed empty infinitely often. Thus we can
use it to do the backwards in discrete-time stationary construction until it
empties, at time (say) $-{\hat N}$; ${\hat N}=\min\{n\ge0: \mathbf{{\hat V}%
}_{-n}=0\}$. Then, we can re-construct the original RA model (starting empty
at time $-{\hat N}$) using the (original untruncated) ${\hat N}$ interarrival
times $(T_{-{\hat N}},T_{-{\hat N}+1},..., T_{-1})$ in lieu of $({\hat
T}_{-{\hat N}},{\hat T}_{-{\hat N}+1},..., {\hat T}_{-1})$, so as to collect
${\hat N}$ re-ordered $S_{n}$ needed in construction of $\mathbf{W}_{0}$ for
the FIFO model.
\end{proof}

\begin{rem}
\textrm{One would expect that the reconstruction of the original RA model in
the above proof is unnecessary, that instead we only need to re-construct the
$\widehat{RA}$ model until we have ${\hat N}$ service initiations from it, as
opposed to ${\hat N}$ service initiations from the original RA model. Although
this might be true, the subtle problem is that the order in which service
times are initiated in the ${ \widehat{RA}}$ model will typically be different
than for the original RA model; they have different arrival processes
(counterexamples are easy to construct). Thus it is not clear how one can
utilize Lemma~\ref{l:main-inequality} and Lemma~\ref{l: service-times-RA} and
so on. One would need to generalize Lemma~\ref{l:main-inequality} to account
for truncated arrival times used in the RA model, but not the FIFO model, in
perhaps a form such as a variation of Equation
(\ref{e:main-inequality-arrival-times}),
\begin{equation}
\label{e:main-inequality-arrival-times-different}P(Q_{F}(t_{n}-)\le
Q_{\widehat{RA}}({\hat t}_{n}-),\ \hbox{for all}\ n\ge0)=1,
\end{equation}
where $\{{\hat t}_{n}\}$ is the truncated renewal process. We did not explore
this further.}
\end{rem}

\section{Infinite server systems and other service disciplines}

In this Section we sketch how one can utilize our FIFO $GI/GI/c$ results to
obtain exact sampling of some other models including the infinite server
queue, and the multi-server queue under other disciplines.

In \cite{B-D} an exact simulation algorithm is presented for simulating from
the stationary distribution of the infinite server queue; the $GI/GI/\infty$.
Here we sketch how to utilize our new FIFO $GI/GI/c$ results to accomplish
this by using a FIFO $GI/GI/c$ model as an upper bound. The $GI/GI/\infty$
model has an infinite number of servers, there is no line, every arrival
enters service immediately upon arrival; the $n^{th}$ customer arrives at time
$t_{n}$ and departs at time $t_{n}+S_{n}$.

For $0<\rho=\lambda/\mu<\infty$, this model is always stable. Let $c$ denote
the smallest integer strictly larger than $\rho$; $c-1\leq\rho<c$. Note that
any $c>\rho$ can be chosen for this construction. A larger value of $c$ will
result in a smaller number of arrivals necessary to detect coalescence, up to
a certain point, since there will be less congestion in the bounding systems
and the customers will leave faster. On the other hand, the actual simulation
time may increase just as a consequence of simulating a $c$ dimensional random
walk. We suggest a rule consistent with square-root staffing, $c=\rho
+\sqrt{\rho}$, since this is well known to trade quality and capacity costs,
which in our setting precisely translate to trading faster coalescence with
cost-per-replication costs (see \cite{H-W}).

Letting $V_{\infty}(t)$ denote the total amount of work in the $GI/GI/\infty$
model, and $V_{c}(t)$ denote the total amount of work in the (necessarily
stable) FIFO $GI/GI/c$ model being fed exactly the same input (of service
times and interarrival times), and both starting initially empty, the
following is easily established:%

\begin{equation}
P(V_{\infty}(t)\le V_{c}(t),\ \hbox{for all}\ t\ge0)=1,
\end{equation}
hence
\begin{equation}
P(V_{\infty}(t_{n}-)\le V_{c}(t_{n}-),\ \hbox{for all}\ n\ge0)=1.
\end{equation}

(Note that both models use the service times in the same order of initiation,
which makes the coupling easy from the start.)

Thus, if, for example $P(T>S)>0$, then the FIFO model will empty and can be
used to detect times when the $GI/GI/\infty$ model will empty. Let $L_{\infty
}(t_{n}-)$ denote the total number of busy servers in the $GI/GI/\infty$ model
as found by $C_{n}$.

Simulating the FIFO model backwards in time in stationarity (using our
previous algorithm), until it first empties, can then be used to detect a time
when the $GI/GI/\infty$ model is empty, and then one can construct it back up
to time $t=0$ to obtain a stationary copy of $V_{\infty}(t_{n}-)$ and of
$L_{\infty}(t_{n}-)$.

Now we consider alternatives disciplines to FIFO for the $GI/GI/c$ model. It
is immediate that when service times are generated only when needed by a
server, the total number of customers in the system process $\{Q(t)\}$ remains
the same under FIFO as under last-in-first-out (LIFO) in which the next
customer to enter service is the one at the bottom of the line, or random
selection next (RS) in which the next customer to enter service from the line
is selected at random by the server. Thus, they all share the same stationary
distribution of $Q(t)$ as $t\to\infty$, as well as the stationary distribution
of $Q(t_{n}-)$ as $n\to\infty$. Let $Q_{0}$ have this limiting (as $n\to
\infty$) distribution. This fact can be used to exactly simulate, for example,
stationary delay $D$ under LIFO or RS (they are not the same as for FIFO). The
method (sketch) is as follows: Simulate a copy of $Q_{0}$, jointly with the
remaining service times of those in service, by assuming FIFO. This represents
the distribution of the system as found in stationarity (at time $0$) by
arrival $C_{0}$. Consider RS for example. If the line is empty, then define
$D_{RS}=0$; $C_{0}$ enters service immediately. Otherwise, place $C_{0}$ in
the line, and continue simulating but now using RS instead of FIFO. As soon as
$C_{0}$ enters service, stop and define $D_{RS}$ as that length of time.

\section{Fork-Join Models}

The RA recursion (\ref{e:RA-RRW-vector}),
\begin{equation}
\label{e:FJ}\mathbf{V}_{n+1} =(\mathbf{V}_{n} +\mathbf{{S}}_{n} -\mathbf{T}%
_{n})^{+},\ n\ge0,
\end{equation}
is actually a special case for the modeling of \textit{Fork-Join} (FJ) queues
(also called \textit{Split and Match}) with $c$ nodes. In an FJ model, each
arrival is a ``job" with $c$ components, the $i^{th}$ component requiring
service at the $i^{th}$ FIFO queue. So upon arrival at time $t_{n}$, the job
splits into its $c$ components to be served. As soon as all $c$ components
have completed service, then and only then, does the job depart. Such models
are useful in manufacturing applications. The $n^{th}$ job ($C_{n}$) thus
arrives with a service time vector attached of the form $\mathbf{S}_{n} =
({S}_{n}(1), . . . ,{S}_{n}(c))$. Let us assume that the vectors are iid, but
otherwise each vector's coordinates can have a general joint distribution; for
then (\ref{e:FJ}) still forms a Markov chain. We will denote this model as the
$GI/GI/c-FJ$ model. The sojourn time of the $i^{th}$ component is given by
$V_{n}(i)+S_{n}(i)$, and thus the sojourn time of the $n^{th}$ job, $C_{n}$,
is given by
\begin{equation}
\label{e:FJ-sojourn}H_{n}=\max_{1\le i\le c}\{V_{n}(i)+S_{n}(i)\}.
\end{equation}
Of great interest is obtaining the limiting distribution of $H_{n}$ as
$n\to\infty$; we denote a r.v. with this distribution as $H^{0}$. FJ models
are notoriously difficult to analyze analytically: Even the special case of
Poisson arrivals and i.i.d. exponential service times is non-trivial because
of the dependency of the $c$ queues through the common arrival process. (A
classic paper is Flatto \cite{Flatto}). In fact when $c\ge3$, only bounds and
approximations are available. As for exact simulation, there is a paper by
Hongsheng Dai \cite{HD}, in which Poisson arrivals and independent exponential
service times are assumed. Because of the continuous-time Markov chain (CTMC)
model structure, the author is able to construct (simulate) the time-reversed
CTMC to use in a coupling from the past algorithm. But with general renewal
arrivals and or general distribution service times, such CTMC methods no
longer can be used.

Our simulation method for the RA model outlined in
Section~\ref{sub:RA-simulation-exact}, however yields an exact copy of $H^{0}$
for the general $GI/GI/c-FJ$ model, under the condition that there exists
$\mathbf{\theta}>\mathbf{0}$, $\mathbf{\theta}\in\mathbb{R}^{c}$ such that
\[
E(\exp(\mathbf{\theta}^{T}(\mathbf{S}_{1}-\mathbf{T}_{1})))<\infty.
\]
First we simulate $\mathbf{V}_{0}^{0}$ exactly using exponential change of
measure method introduced in \cite{B-C} (we use the same technique for
multidimensional simulation in Algorithm \ref{sub:algorithm-main-sketch}),
then simulate a vector of service times $\mathbf{S}=(S(1),\ldots,S(c))$
independently and set
\[
H^{0}=\max_{1\le i\le c}\{V_{0}^{0}(i)+S(i)\}.
\]


Even when the service time components within $\mathbf{S}$ are independent, or
the case when service time distributions are assumed to have a finite moment
generating function (in a neighborhood of the origin), such results are new
and non-trivial.

\section{The case when $P(T>S)=0$: Harris recurrent regeneration}

\label{s:Harris-Chain} For a stable FIFO $GI/GI/c$ queue, the stability
condition can be re-written as $E(T_{1}+\cdots+ T_{c})> E(S)$, which implies
also that $P(T_{1}+\cdots+ T_{c}>S)>0$. Thus assuming that $P(T>S)>0$ is not
necessary for stability. When $P(T>S)=0$, the system will never empty again
after starting, and so using consecutive visits to $\mathbf{0}$ as
regeneration points is not possible. But the system does regenerate in a more
general way via the use of Harris recurrent Markov chain theory; see
\cite{KS-1988} for details and history of this approach. The main idea is that
while the system will not empty infinitely often, the number in system process
$\{Q_{F}(t_{n}-):n\ge0\}$ will visit an integer $1\le j\le c-1$ infinitely often.

For illustration here, we will consider the $c=2$ case (for the general case
$c\ge2$ the specific regeneration points analogous to what we present here are
carefully given in Equation (4.6) on page 396 of \cite{KS-1988}). Let assume
that $1<\rho<2$. (Note that if $\rho<1$, then equivalently $E(T)>E(S)$ and so
$P(T>S)>0$; that is why we rule out $\rho<1$ here.) We now assume that
$P(T>S)=0$. This implies that for ${\underline{s}}\triangleq\inf\{s>0:
P(S>s)>0\}$ and ${\overline t}\triangleq\sup\{t>0: P(T>t)>0\}$, we must have
$0<{\overline t}<{\underline{s}}<\infty$. It is shown in \cite{KS-1988} that
for $\epsilon>0$ sufficiently small, the following event will happen
infinitely often (in $n$) with probability $1$,%

\begin{equation}
\label{e:RA-Pre-reg}\{Q_{RA}(t_{n}-)=1,\ V_{n}(1)=0,\ V_{n}(2)\le
\epsilon,\ T_{n}>\epsilon,\ U_{n}=1\}.
\end{equation}

If $n$ is such a time, then at time $n+1$, we have%

\begin{equation}
\label{e:RA-regeneration-point}\{Q_{RA}(t_{n+1}-)=1,\ V_{n+1}(2)=0,\ V_{n+1}%
(1)=(S_{n}-T_{n})\}.
\end{equation}

The point is that $C_{n}$ finds one server (server 1) empty, and the other
queue with only one customer in it, and that customer is in service with a
remaining service time $\le\epsilon$. $C_{n}$ then enters service at node $1$
with service time $S_{n}$; but since $T_{n}>\epsilon$, $C_{n+1}$ arrives
finding the second queue empty, and the first server has remaining service
time $S_{n}-T_{n}$ conditional on $T_{n}>\epsilon$. Under the coupling of
Lemma~\ref{l:main-inequality}, the same will be so for the FIFO model (see
Remark~\ref{rem:coupling-work-application} below): At such a time $n$,
\begin{equation}
\label{e:FIFO-Pre-reg}\{Q_{F}(t_{n}-)=1,\ W_{n}(1)=0,\ W_{n}(2)\le
\epsilon,\ T_{n}>\epsilon\},
\end{equation}
and at time $n+1$ we have
\begin{equation}
\label{e:FIFO-regeneration-point}\{Q_{F}(t_{n+1}-)=1,\ W_{n}(1)=0,\ W_{n}%
(2)=(S_{n}-T_{n})\}.
\end{equation}

Equations (\ref{e:RA-regeneration-point}) and (\ref{e:FIFO-regeneration-point}%
) define positive recurrent regeneration points for the two models (at time
$n+1$); the consecutive times at which regenerations occur form a
(discrete-time) positive recurrent renewal process, see \cite{KS-1988}.

To put this to use, we change the stopping time $N$ given in
(\ref{eq-stopping-time}) to:%

\begin{align}
\label{e:main-stopping-time-regenerate}N+1  &  =\min\{ n\ge1: Q_{RA}%
^{0}(t_{-(n+1) }-)=1,\ V_{-(n+1)}^{0}(1)=0,\\
&  V_{-(n+1)}^{0}(2)\le\epsilon,\ T_{-(n+1)} > \epsilon,\ U_{-(n+1)}=1
\}.\nonumber
\end{align}
Then we do our reconstructions for the algorithm in
Section~\ref{sub:algorithm-main-sketch} by starting at time $-N$, with both
models starting with the same starting value
\begin{equation}
\label{e:RA-regeneration-point-sim}\{Q_{RA}(t_{-N}-)=1,\ V_{-N}^{0}%
(2)=0,\ V_{-N}^{0}(1)=(S_{-(N+1)}-T_{-(N+1)})\ | \ T_{-(N+1)}>\epsilon\}
\end{equation}

\begin{equation}
\label{e:FIFO-regeneration-point-sim}\{Q_{F}(t_{-N}-)=1,\ W_{-N}%
(1)=0,\ W_{-N}(2)=(S_{-(N+1)}-T_{-(N+1)})\ | \ T_{-(N+1)}>\epsilon\}.
\end{equation}

\begin{rem}
\label{rem:coupling-work-application}

\textrm{The service time used in (\ref{e:RA-regeneration-point-sim}) and
(\ref{e:FIFO-regeneration-point-sim}) for coupling via
Lemma~\ref{l: service-times-RA}, $S_{-(N+1)},$ is in fact identical for both
systems because (subtle): At time $-(N+1)$, both systems have only one
customer in system, and thus total work is in fact equal to the remaining
service time; so we use Equation (\ref{e:main-inequality-work-arrival-times})
to conclude that both remaining service times (even if different) are
$\leq\epsilon$ (e.g., that is why (\ref{e:FIFO-Pre-reg}) follow from
(\ref{e:RA-Pre-reg})). Meanwhile, $C_{-(N+1)}$ enters service immediately
across both systems, so it is indeed the same service time $S_{-(N+1)}$ used
for both for this initiation. Coalescence is detected in finite expected time
because of the positive recurrence property underlying the definition of the
regeneration points from (\ref{e:RA-regeneration-point}) and
(\ref{e:FIFO-regeneration-point})}.
\end{rem}

\appendix

\section{Appendices}

\subsection{Detailed algorithm steps in Section
\ref{sub:algorithm-main-sketch}}

\label{appendix-simulation-algo}


To simulate the process $\{(\mathbf{R}_{n}^{(r)},\mathbf{V}_{-n}^{0}):0\leq
n\leq N\}$ with the time $N$ defined in (\ref{eq-stopping-time}) as
\[
N=\inf\left\{  n\geq0:\mathbf{V}_{-n}^{0}=\max_{k\geq n}\mathbf{R}_{k}%
^{(r)}-\mathbf{R}_{n}^{(r)}=\mathbf{0}\right\}  ,
\]
we must sample the running time maxima (entry by entry) of the $c$-dimensional
random walk
\[
\mathbf{R}_{n}^{(r)}=\sum_{i=1}^{n}\mathbf{\Delta}_{-i}=\sum_{i=1}^{n}%
(\tilde{\mathbf{S}}_{-i}-\mathbf{T}_{-i})~~~~n\geq0.
\]
We will find a sequence of random times $\{N_{n}:n\geq1\}$ such that
$\max_{n\leq k\leq N_{n}}\mathbf{R}_{k}^{(r)}\geq\max_{k\geq N_{n}}%
\mathbf{R}_{k}^{(r)}$. Hence, we will be able to find the running time maxima
by only sampling the random walk on a finite time interval, i.e., $N_{n}$ is
such that
\[
\max_{k\geq n}\mathbf{R}_{k}^{(r)}=\max_{n\leq k\leq N_{n}}\mathbf{R}%
_{k}^{(r)}.
\]
To achieve this, we first decompose the random walk into two random walks and
then construct a sequence of \textquotedblleft milestone\textquotedblright%
\ events for each of these two random walks to detect $N_{n}$'s. We will
elaborate the detailed implementations in the following context.

Because of the stability condition $\rho=\lambda/\mu<c$, we can find some
value $a\in(1/\mu,c/\lambda)$.
For any $n\ge0$, define
\begin{align}
\mathbf{X}_{-n}  &  =\sum_{j=1}^{n}\left(  S_{-j}-a\right)  \mathbf{U}_{-j}
,\label{e:X-def}\\
\mathbf{Y}_{-n}  &  =\sum_{j=1}^{n}\left(  a\mathbf{U}_{-j}-\mathbf{T}%
_{-j}\right)  , \label{e:Y-def}%
\end{align}
hence $\mathbf{R}^{(r)}_{n}=\sum_{j=1}^{n}\mathbf{\Delta}_{-j}=\mathbf{X}%
_{-n}+\mathbf{Y}_{-n}$ and $\max_{k\ge n}\mathbf{R}_{k}^{(r)}=\max_{k\ge
n}(\mathbf{X}_{-n}+\mathbf{Y}_{-n})$.

For all $n\ge0$, let
\begin{align}
N_{n}^{X}  &  =\inf\{n^{\prime}\ge n:\max_{k\ge n^{\prime}}\mathbf{X}_{-k}%
\le\mathbf{X}_{-n}\},\label{running-finite-check-X}\\
N_{n}^{Y}  &  =\inf\{n^{\prime}\ge n:\max_{k\ge n^{\prime}}\mathbf{Y}_{-k}%
\le\mathbf{Y}_{-n}\},\label{running-finite-check-Y}\\
N_{n}  &  =\max\{N_{n}^{X},N_{n}^{Y}\} . \label{running-finite-check}%
\end{align}
Then, by the definitions above,
\[
\max_{k\ge N_{n}}\mathbf{R}_{k}^{(r)}\le\max_{k\ge N_{n}}\mathbf{X}_{-k}%
+\max_{k\ge N_{n}}\mathbf{Y}_{-k}\le\mathbf{X}_{-n}+\mathbf{Y}_{-n}%
=\mathbf{R}_{n}^{(r)}.
\]
Therefore, to get the running-time maximum $\max_{k\ge n}\mathbf{R}^{(r)}_{k}$
for each $n\ge0$, we only need to sample the random walk from step $n$ to
$N_{n}$, because
\[
\max_{k\ge n}\mathbf{R}_{k}^{(r)}=\max\{\max_{n\le k\le N_{n}}\mathbf{R}%
_{k}^{(r)},\max_{n\ge N_{n}}\mathbf{R}_{k}^{(r)}\}=\max_{n\le k\le N_{n}%
}\mathbf{R}^{(r)}_{k}.
\]

Next we describe how to sample $N_{n}$ along with the multi-dimensional random
walks $\{\mathbf{X}_{-n}:n\ge0\}$ and $\{\mathbf{Y}_{-n}:n\ge0\}$.

\subsubsection{Simulation algorithm for the process $\{\mathbf{Y}_{-n}%
:n\ge0\}$}

\label{section-sample-Y} We first consider simulating the $c$-dimensional
random walk $\{\mathbf{Y}_{-n}:n\geq0\}$ with $\mathbf{Y}_{0}=\mathbf{0}$. For
each $j\geq1$, $E(a\mathbf{U}_{-j}-\mathbf{T}_{-j})<\mathbf{0}$, we can
simulate the running time maximum $\max_{k\geq n}\mathbf{Y}_{-k}$ jointly with
the path $\{\mathbf{Y}_{-k}:0\leq k\leq n\}$ via the method developed in
\cite{B-C}, with the following assumptions.

\textit{Assumption (A1):} There exits $\mathbf{\theta}>\mathbf{0}$,
$\mathbf{\theta}\in\mathbb{R}^{c}$ such that
\[
E\exp\left(  \mathbf{\theta}^{T}(a\mathbf{U}_{-j}-\mathbf{T}_{-j})\right)
<\infty.
\]

\textit{Assumption (A1b):} Suppose that in every dimension $i=1,\ldots,c$,
there exists $\theta^{*}\in(0,\infty)$ such that
\[
\phi_{i}(\theta^{*}):=\log E\exp\left(  \theta^{*}\left(  aI(U_{-j}%
=i)-T_{-j}\right)  \right)  =0.
\]
Because for each $j\ge1$, $aI(U_{-j}=i)-T_{-j}$ are marginally identically
distributed across $i$, so $\theta^{*}$ would work for all $i=1,\ldots,c$.

\begin{rem}
In our setting, since $\mathbf{U}_{-j}$ is bounded, assumption (A1) always
holds. Assumption (A1b) is known as Cramer's condition in the large deviations
literature and it is a strengthening of assumption (A1). We shall explain
briefly at the end of this section that it is possible to relax this
assumption to (A1) by modifying the algorithm a bit without affecting the
exactness/computational effort of the algorithm. For the moment we continue to
describe the main algorithmic idea under assumption (A1b).
\end{rem}


For any $\mathbf{s}\in\mathbb{R}^{c}$ and $\mathbf{b}\in\mathbb{R}^{c}_{+}$
define
\begin{align}
\label{e:go-above-time-def} &  T_{\mathbf{b}}=\inf\{n\ge0:Y_{-n}%
(i)>b(i)~~\mbox{for some }i\in\{1,\ldots,c\}\},\\
&  T_{-\mathbf{b}}=\inf\{n\ge0:Y_{-n}(i)<-b(i)~~\mbox{for all }i=1,\ldots
,c\},\\
&  P_{\mathbf{s}}(\cdot)=P(\cdot\vert\mathbf{Y}_{0}=\mathbf{s}).
\end{align}
We will use these definitions in Algorithm {\small LTGM} given in Section
\ref{sec-ltgm}.


We next construct a sequence of upward and downward \textquotedblleft
milestone\textquotedblright\ events for this multidimensional random walk. The
construction is completely analogous to the classical ladder height
decomposition of one dimensional random walks, we introduce a parameter, $m$,
in order to facilitate a certain acceptance / rejection step to be explained
in the next subsection. Let
\begin{equation}
m=\lceil\log(c)/\theta^{\ast}\rceil. \label{e:m-value-choose}%
\end{equation}
Define $D_{0}=0$ and $\Gamma_{0}=\infty$. For $k\geq1$, let
\begin{align}
D_{k}  &  =\inf\{n\geq D_{k-1}\vee\Gamma_{k-1}I\left(  \Gamma_{k-1}%
<\infty\right)  :Y_{-n}(i)<Y_{-D_{k-1}}%
(i)-m~~\mbox{for all }i\},\label{e:upward-milestone}\\
\Gamma_{k}  &  =\inf\{n\geq D_{k}:Y_{-n}(i)>Y_{-D_{k}}%
(i)+m~~\mbox{for some }i\}. \label{e:downward-milestone}%
\end{align}
Note that by convention, $\Gamma_{k}I\left(  \Gamma_{k}<\infty\right)  =0$ if
$\Gamma_{k}=\infty$ for any $k\geq0$. We let $\mathbf{B}\in\mathbb{R}^{c}$,
initially set as $(\infty,\ldots,\infty)^{T}\in\mathbb{R}^{c}$, to be the
running time upper bound of process $\{\mathbf{Y}_{-n}:n\geq0\}$. Let
$\mathbf{m}=m\mathbf{f}$, where $\mathbf{f}=(1,\ldots,1)^{T}$ as defined in
Section \ref{FIFO-model}. From the construction of \textquotedblleft
milestone" events in (\ref{e:upward-milestone}) and
(\ref{e:downward-milestone}), we know that if $\Gamma_{k}=\infty$ for some
$k\geq1$, the process will never cross over the level $\mathbf{Y}_{-D_{k}%
}+\mathbf{m}$ after $D_{k}$ coordinate-wise, i.e., for $i=1,\ldots,c$,
\[
Y_{-n}(i)\leq Y_{-D_{k}}(i)+m,~~~~\forall n\geq D_{k}.
\]
Hence, in this case we update the upper bound vector $\mathbf{B}%
=\mathbf{Y}_{-D_{k}}+\mathbf{m}$.

\paragraph{Global maximum simulation}

\label{sec-ltgm}

Define
\begin{equation}
\Lambda=\inf\{D_{k}:\Gamma_{k}=\infty,k\geq1\}. \label{e:delta-def}%
\end{equation}
By the construction of \textquotedblleft milestone\textquotedblright\ events,
for all $n\geq\Lambda$
\[
\mathbf{Y}_{-n}\leq\mathbf{Y}_{-\Lambda}+\mathbf{m}<\mathbf{0}=\mathbf{Y}%
_{0}.
\]
Hence, we can evaluate the global maximum level of the process $\{\mathbf{Y}%
_{-n}:n\geq0\}$ to be
\begin{equation}
\mathbf{M}_{0}:=\max_{k\geq0}\mathbf{Y}_{-k}=\max_{0\leq k\leq\Lambda
}\mathbf{Y}_{-k}, \label{e:global-max}%
\end{equation}
and we give the detailed sampling procedure in the following algorithm. The
algorithm has elements, such as sampling from $P_{\mathbf{0}}(T_{\mathbf{m}%
}<\infty)$, which will be explained in the sequel.

\bigskip
\noindent\textbf{Algorithm} {\small LTGM}: simulate global maximum of
$c$-dimensional process $\{\mathbf{Y}_{-n}:n\ge0\}$ jointly with the sub-path
and the subsequence of ``milestone" events.\label{alg1}\newline

\noindent Input: $a\in(1/\mu,c/\lambda)$ satisfies assumption (A1b), $m$ as in
(\ref{e:m-value-choose}).

\begin{enumerate}
\item (\textit{Initialization}) Set $n=0$, $\mathbf{Y}_{0}=\mathbf{0}$,
$\mathbf{D}=[0]$, $\mathbf{\Gamma}=[\infty]$, $\mathbf{L}=\mathbf{0}$ and
$\mathbf{B}=\infty\mathbf{f}$. \label{global-max-1}

\item Generate $U\sim Unif\{1,\ldots,c\}$ and let $\mathbf{U}=(I(U=1),\ldots
,I(U=c))^{T}$. Independently sample $T\sim A$ and let $\mathbf{T}=T\mathbf{f}%
$. Set $n=n+1$, $\mathbf{Y}_{-n}=\mathbf{Y}_{-(n-1)}+a\mathbf{U}-\mathbf{T}$,
$U_{-n}=U$ and $T_{-n}= T$.\label{global-max-2}

\item If there is some $1\le i\le c$ such that $Y_{-n}(i)\ge L(i)-m$, then go
to Step \ref{global-max-2}; otherwise set $\mathbf{D}=[\mathbf{D},n]$ and
$\mathbf{L}=\mathbf{Y}_{-n}$.

\item Independently sample $J\sim Ber\left(  P_{\mathbf{0}}\left(
T_{\mathbf{m}}<\infty\right)  \right)  $. \label{global-max-4}%
\label{global-max-3}

\item If $J=1$, simulate a new conditional path $\{\left(  \mathbf{y}%
_{-k},u_{-k},t_{-k}\right)  :1\le k\le T_{\mathbf{m}}\}$ with $\mathbf{y}%
_{0}=\mathbf{0}$, following the conditional distribution of $\{\mathbf{Y}%
_{-k}:0\le k\le T_{\mathbf{m}}\}$ given $T_{\mathbf{m}}<\infty$. Set
$\mathbf{Y}_{-(n+k)}=\mathbf{Y}_{-n}+\mathbf{y}_{-k}$, $U_{-(n+k)}= u_{-k}$,
$T_{-(n+k)}= t_{-k}$ for $1\le k\le T_{m}$. Set $n= n+T_{\mathbf{m}}$,
$\mathbf{\Gamma}=[\mathbf{\Gamma},n]$.
Go to Step \ref{global-max-2}.\label{global-max-5}

\item If $J=0$, set $\Lambda=n$, $\mathbf{\Gamma}=[\mathbf{\Gamma},\infty]$
and $\mathbf{B}=\mathbf{L}+\mathbf{m}$.\label{global-max-6}


\item Output $\left\{  \left(  \mathbf{Y}_{-k},U_{-k},T_{-k}\right)  :1\le
k\le\Lambda\right\}  $, $\mathbf{D}$, $\mathbf{\Gamma}$ and global maximum
$\mathbf{M}_{0}=\max_{0\le k\le\Lambda}\mathbf{Y}_{-k}$.
\end{enumerate}

\bigskip

Now we explain how to execute Steps \ref{global-max-4} and \ref{global-max-5}
in the previous algorithm. The procedure is similar to the multi-dimensional
procedure given in \cite{B-C}, so we describe it briefly here. As
$P_{\mathbf{0}}(\cdot)$ denotes the canonical probability, we let
$P^{*}_{\mathbf{0}}(\cdot)=P_{\mathbf{0}}(\cdot|T_{\mathbf{m}}<\infty)$. Our
goal is to simulate from the conditional law of $\left\{  \mathbf{Y}_{-k}:0\le
k\le T_{\mathbf{m}}\right\}  $ given that $T_{\mathbf{m}}<\infty$ and
$\mathbf{Y}_{0}=\mathbf{0}$, i.e., to simulate from $P_{\mathbf{0}}^{*}$. We
will use acceptance/rejection by letting $P^{\prime}_{\mathbf{0}}(\cdot)$
denote the proposal distribution. A typical element $\omega^{\prime}$ sampled
under $P_{\mathbf{0}}^{\prime}(\cdot)$ is of the form $\omega^{\prime
}=((\mathbf{Y}_{-k}:k\ge0),index)$, where $index\in\{1,\cdots,c\}$ and it
indicates the direction we pick to do exponential tilting; it is the
coordinate in which we change its measure to increase the chance of its
hitting the upward ``milestone" as defined in (\ref{e:downward-milestone}).
Given the value of $index$, the process $(\mathbf{Y}_{-k}:k\ge0)$ remains a
random walk. We now describe $P_{\mathbf{0}}^{\prime}$ by explaining how to
sample $\omega^{\prime}$. First,
\begin{equation}
\label{e:tilt-direction}P_{\mathbf{0}}^{\prime}\left(  index=i\right)
=w_{i}:=\frac{1}{c}.
\end{equation}
Then, conditioning on $index=i$, for every set $A\in\sigma\left(
\{\mathbf{Y}_{-k}:0\le k\le n\}\right)  $,
\begin{equation}
\label{exponential-tilting}P_{\mathbf{0}}^{\prime}\left(  A\vert
index=i\right)  =E_{\mathbf{0}}\left(  \exp\left(  \theta^{*}Y_{-n}(i)\right)
I_{A}\right)  .
\end{equation}

To obtain the induced distribution for $U$ and $T$, we study the moment
generating function induced by definition (\ref{exponential-tilting}). Given
$\mathbf{\eta}\in\mathbb{R}^{c}$ in a neighborhood of the origin,
\[
\frac{E_{\mathbf{0}}\exp\left(  \mathbf{\eta}^{T}(a\mathbf{U}-\mathbf{T}%
)+\theta^{*}e_{i}^{T}(a\mathbf{U}-\mathbf{T})\right)  }{E_{\mathbf{0}}%
\exp\left(  \theta^{*}e_{i}^{T}(a\mathbf{U}-\mathbf{T})\right)  }%
=\frac{E_{\mathbf{0}}\exp\left(  (\mathbf{\eta}+\theta^{*}e_{i})^{T}%
a\mathbf{U}\right)  }{E_{\mathbf{0}}\exp\left(  \theta^{*}e_{i}^{T}%
a\mathbf{U}\right)  }\cdot\frac{E_{\mathbf{0}}\exp\left(  -(\mathbf{\eta
}+\theta^{*}e_{i})^{T}\mathbf{T}\right)  }{E_{\mathbf{0}}\exp\left(
-\theta^{*}e_{i}^{T}\mathbf{T}\right)  }.
\]
The previous expression indicates that under $P_{\mathbf{0}}^{\prime}(\cdot)$,
$T$ and $U$ are independent. Moreover, we have
\[
E_{\mathbf{0}}\exp\left(  \theta^{*}e_{i}^{T}a\mathbf{U}\right)  =\frac
{\exp({\theta^{*}a})+c-1}{c}.
\]
Therefore,
\begin{equation}
\label{e:tilt-u-dist}P_{\mathbf{0}}^{\prime}(U=j\vert index=i)=%
\begin{cases}
\frac{\exp\left(  \theta^{*}a\right)  }{\exp\left(  \theta^{*}a\right)  +c-1}
& \text{if } j =i\\
\frac{1}{\exp\left(  \theta^{*}a\right)  +c-1} & \text{if } j\neq i
\end{cases}
.
\end{equation}
On the other hand, conditional on $index=i$, the distribution of a generic
interarrival time $T$ is obtained by exponential tilting such that
\begin{align}
\label{e:tilt-t-dist}dP_{\mathbf{0}}(T\vert index=i)  &  =dP_{\mathbf{0}%
}(T)\cdot\frac{\exp(-\theta^{*}T)}{E_{\mathbf{0}}\exp(-\theta^{*}%
T)}\nonumber\\
&  =dP_{\mathbf{0}}(T)\cdot\frac{\exp(a\theta^{*})+c-1}{c\exp(\theta^{*}T)},
\end{align}
where the second equation follows from assumption (A1b).

Following assumption (A1b) and because $Var(aI(U_{-j}=i)-T_{-j})>0$, by
convexity,
\begin{align*}
E_{\mathbf{0}}^{\prime}\left(  Y_{-n}(index)\right)   &  =\sum_{i=1}%
^{c}E_{\mathbf{0}}\left(  Y_{-n}(i)\exp\left(  \theta^{*}Y_{-n}(i)\right)
\right)  P_{\mathbf{0}}^{\prime}\left(  index=i\right) \\
&  =\frac{1}{c}\sum_{i=1}^{c}\frac{d\phi_{i}\left(  \theta^{*}\right)
}{d\theta}>0,
\end{align*}
so $Y_{-n}\left(  index\right)  \to\infty$ as $n\to\infty$ almost surely under
$P_{\mathbf{0}}^{\prime}(\cdot)$, hence $T_{\mathbf{m}}<\infty$ with
probability one under $P_{\mathbf{0}}^{\prime}(\cdot)$. Now, to verify that
$P_{\mathbf{0}}(\cdot)$ is a valid proposal for acceptance/rejection method,
we must verify that $dP_{\mathbf{0}}^{*}/dP_{\mathbf{0}}^{\prime}$ is bounded
by a constant, i.e.,
\begin{align*}
&  \frac{dP_{0}^{*}}{dP_{0}^{\prime}}\left(  \mathbf{Y}_{-k}:0\le k\le
T_{\mathbf{m}}\right) \\
=  &  \frac{1}{P_{0}\left(  T_{\mathbf{m}}<\infty\right)  }\times\frac{dP_{0}%
}{dP_{0}^{\prime}}\left(  \mathbf{Y}_{-k}:0\le k\le T_{\mathbf{m}}\right) \\
=  &  \frac{1}{P_{0}\left(  T_{\mathbf{m}}<\infty\right)  }\times\frac{1}%
{\sum_{i=1}^{c}w_{i}\exp\left(  \theta^{*}Y_{-T_{\mathbf{m}}}(i)\right)  }\\
\le &  \frac{1}{P_{0}\left(  T_{\mathbf{m}}<\infty\right)  }\times\frac
{c}{\exp\left(  \theta^{*}m\right)  }\\
<  &  \frac{1}{P_{0}\left(  T_{\mathbf{m}}<\infty\right)  },
\end{align*}
where the last inequality is guaranteed by (\ref{e:m-value-choose}). So,
acceptance/rejection is valid.

Moreover, the overall probability of accepting the proposal is precisely
$P_{\mathbf{0}}(T_{\mathbf{m}}<\infty)$. Thus, we not only execute Step
\ref{global-max-5}, but simultaneously also Step \ref{global-max-4}. We use
this acceptance/rejection method to replace Steps \ref{global-max-4} and
\ref{global-max-5} in Algorithm {\small LTGM} as follows: \bigskip

\begin{enumerate}
\item[4'] Sample $\left\{  \left(  \mathbf{y}_{-k},u_{-k},t_{-k}\right)  :0\le
k\le T_{\mathbf{m}}\right)  \}$ with $\mathbf{y}_{0}=\mathbf{0}$ from
$P_{0}^{\prime}\left(  \cdot\right)  $ as indicated via
(\ref{e:tilt-direction}), (\ref{e:tilt-u-dist}) and (\ref{e:tilt-t-dist}).
Sample a Bernoulli $J$ with success probability
\[
\frac{c}{\sum_{i=1}^{c}\exp\left(  \theta^{*}y_{-T_{\mathbf{m}}}(i)\right)
}.
\]

\item[5'] If $J=1$, set $\mathbf{Y}_{-(n+k)}=\mathbf{Y}_{-n}+\mathbf{y}_{-k}$,
$U_{-(n+k)}=u_{-k}$, $T_{-(n+k)}=t_{-k}$ for $1\le k\le T_{\mathbf{m}}$. Set
$n=n+T_{\mathbf{m}}$ and $\mathbf{\Gamma}=[\mathbf{\Gamma},n]$. Go to Step
\ref{global-max-2}.
\end{enumerate}

\bigskip

\paragraph{Simulate $\{\mathbf{Y}_{-n}:n\ge0\}$ jointly with ``milestone"
events}

In this section we provide an algorithm to sequentially simulate the
multi-dimensional random walk $\{\mathbf{Y}_{-n}:n\ge0\}$ along with its
downward and upward ``milestone" events as defined in
(\ref{e:upward-milestone}) and (\ref{e:downward-milestone}). We first extend
Lemma 3 in \cite{B-S} to multi-dimensional version as follows.



\begin{lemm}
\label{lm:upper-bound} Let $0<\mathbf{a}<\mathbf{b}\le\infty\mathbf{f}$
(coordinate-wise) and consider any sequence of bounded positive measurable
functions $f_{k}:\mathbb{R}_{c\times(k+1)}\to[0,\infty)$,
\begin{align*}
&  E_{0}\left(  f_{T_{-\mathbf{a}}}(\mathbf{Y}_{0},\cdots,\mathbf{Y}%
_{-T_{-\mathbf{a}}})\vert T_{\mathbf{b}}=\infty\right) \\
=  &  \frac{E_{0}\left(  f_{T_{-\mathbf{a}}}(\mathbf{Y}_{0},\cdots
,\mathbf{Y}_{-T_{-\mathbf{a}}})\cdot I(Y_{-j}(i)\le b(i),0\le j<
T_{-\mathbf{a}},1\le i\le c)\right)  \cdot P_{\mathbf{Y}_{-T_{-\mathbf{a}}}%
}(T_{\mathbf{b}}=\infty)}{P_{0}\left(  T_{\mathbf{b}}=\infty\right)  }.
\end{align*}
Therefore, if $P_{0}^{**}(\cdot):=P_{0}(\cdot\vert T_{\mathbf{b}}=\infty)$,
then
\begin{equation}
\label{e:upper-bound}\frac{dP_{0}^{**}}{dP_{0}}=\frac{I\left(  Y_{-j}(i)\le
b(i),\forall j<T_{-\mathbf{a}},1\le i\le c\right)  \cdot P_{\mathbf{Y}%
_{-T_{-\mathbf{a}}}}(T_{\mathbf{b}}=\infty)}{P_{0}(T_{\mathbf{b}}=\infty)}%
\le\frac{1}{P_{0}\left(  T_{\mathbf{b}}=\infty\right)  }.
\end{equation}

\end{lemm}

Lemma \ref{lm:upper-bound} enables us to sample a downward patch by using the
acceptance/rejection method with the nominal distribution $P_{0}$ as proposal.
Suppose our current position is $\mathbf{Y}_{-D_{j}}$ (for some $j\ge1$) and
we know that the process will never go above the upper bound $\mathbf{B}$
(coordinate-wise). Next we simulate the path up to time $D_{j+1}$. If we can
propose a downward patch $\left(  \mathbf{y}_{-1},\cdots,\mathbf{y}_{-T_{-m}%
}\right)  :=\left(  \mathbf{Y}_{-1},\cdots,\mathbf{Y}_{-T_{-\mathbf{m}}%
}\right)  $, under the unconditional probability given $\mathbf{y}%
_{0}=\mathbf{0}$ and $\mathbf{y}_{-k}\le\mathbf{m}$ for $1\le k\le
T_{-\mathbf{m}}$, then we accept it with probability $P_{0}\left(
T_{\mathbf{\sigma}}=\infty\right)  $, where $\mathbf{\sigma}=\mathbf{B}%
-\mathbf{Y}_{-D_{j}}-\mathbf{y}_{-T_{-\mathbf{m}}}$. A more efficient way to
sample is to sequentially generate $\left(  \mathbf{y}_{-1},\cdots
,\mathbf{y}_{-\Lambda}\right)  $ with $\mathbf{y}_{0}=\mathbf{0}$ as long as
$\mathbf{m}_{0}:=\max_{0\le k\le\Lambda}\mathbf{y}_{-k}\le\mathbf{m}$
coordinate-wise, then concatenate the sequence to previously sampled subpath.
We give the efficient implementation procedure in the next algorithm.\newline%
\bigskip

\noindent\textbf{Algorithm} {\small LTRW}: continue to sample the process
$\{(\mathbf{Y}_{-k},U_{-k},T_{-k}):0\le k\le n\}$ jointly with the partially
sampled ``milestone" event lists $\mathbf{D}$ and $\mathbf{\Gamma}$, until a
stopping criteria is met.\label{alg2}\newline

\noindent Input: $a$, $m$, previously sampled partial process $\{(\mathbf{Y}%
_{-j},U_{-j},T_{-j}):0\le j\le l\}$, partial ``milestone" sequences
$\mathbf{D}$ and $\mathbf{\Gamma}$, and stopping criteria $\mathcal{H}%
$.\newline(Note that if there is no previous simulated random walk, we
initialize $l=0$, $\mathbf{D}=[0]$ and $\mathbf{\Gamma}=[\infty]$.)

\begin{enumerate}
\item Set $n=l$. If $n=0$, call Algorithm {\small LTGM} to get $\Lambda$,
$\{(\mathbf{Y}_{-k},U_{-k},T_{-k}):0\le k\le\Lambda\}$, $\mathbf{D}$ and
$\mathbf{\Gamma}$. Set $n=\Lambda$.

\item While the stopping criteria $\mathcal{H}$ is not satisfied,

\begin{enumerate}
\item Call Algorithm {\small LTGM} to get $\tilde{\Lambda}$, $\{(\tilde
{\mathbf{Y}}_{-j},\tilde{U}_{-j},\tilde{T}_{-j}):0\le j\le\tilde{\Lambda}\}$,
$\tilde{\mathbf{D}}$, $\tilde{\mathbf{\Gamma}}$ and $\tilde{\mathbf{M}}_{0}%
$.\label{sim-rw-3}

\item If $\tilde{\mathbf{M}}_{0}\le\mathbf{m}$, accept the proposed sequence
and concatenate it to the previous sub-path, i.e., set $\mathbf{Y}%
_{-(n+j)}=\mathbf{Y}_{-n}+\tilde{\mathbf{Y}}_{-j}$, $U_{-(n+j)}=\tilde{U}%
_{-j}$, $T_{-(n+j)}=\tilde{T}_{-j}$ for $1\le j\le\tilde{\Lambda}$. Update the
sequences of ``milestone" events to be $\mathbf{D}=[\mathbf{D},n+\tilde
{\mathbf{D}}(2:\text{end})]$, $\mathbf{\Gamma}=[\mathbf{\Gamma},n+\tilde
{\mathbf{\Gamma}}(2:\text{end})]$ and set $n=n+\tilde{\Lambda}$.
\end{enumerate}

\item Output $\{(\mathbf{Y}_{-k},U_{-k},T_{-k}):0\le k\le n\}$ with updated
``milestone" event sequences $\mathbf{D}$ and $\mathbf{\Gamma}$.
\end{enumerate}

\bigskip

For $n\ge0$, define
\begin{align}
d_{1}(n)  &  =\inf\{D_{k}\ge n:\mathbf{Y}_{-D_{k}}\le\mathbf{Y}_{-n}\}
,\label{e:Y-d1}\\
d_{2}(n)  &  =\inf\{D_{k}> d_{1}(n):\Gamma_{k}=\infty\} , \label{e:Y-d2}%
\end{align}
and $d_{2}(n)$ is an upper bound of $N_{n}^{Y}$ defined in
(\ref{running-finite-check-Y}) because
\[
\max_{k\ge d_{2}(n)}\mathbf{Y}_{-k}\le\mathbf{Y}_{-d_{2}(n)}+\mathbf{m}%
<\mathbf{Y}_{-d_{1}(n)}\le\mathbf{Y}_{-n}.
\]

\begin{rem}
Since assumption (A1b) is a strengthening of assumption (A1), we can
accommodate our algorithms under the general assumption (A1). The
implementation details are the same as that mentioned in the remark section on
page 15 of \cite{B-C}.
\end{rem}

\subsubsection{Simulation algorithm for the process $\{\mathbf{X}_{-n}%
:n\ge0\}$}

\label{section-sample-X} Recall from (\ref{e:X-def}) that for $n\ge0$,
\[
X_{-n}(i)=\sum_{j=1}^{n}(S_{-j}-a)I(U_{-j}=i)~~~~\mbox{for }i=1,\ldots,c.
\]
Define
\begin{align}
&  N_{k}(i)=\sum_{j=1}^{k}I\left(  U_{-j}=i\right)  ,
\label{e:auxiliary-index}\\
&  L_{n}(i)=\inf\left\{  k\ge0:N_{k}(i)=n\right\}  ~~(L_{0}%
(i)=0),\label{e:universal-index}\\
&  \hat{S}_{-n}^{(i)}=S_{-L_{n}(i)} , \label{e:service-mapping}%
\end{align}
for $k\ge0$, $n\ge0$ and $i=1,\ldots,c$. $N_{k}(i)$ denotes the total number
of customers routed to server $i$ among the first $k$ arrivals counting
backwards in time. $L_{n}(i)$ denotes the index of the $n$-th customer that
gets routed to server $i$ in the common arrival stream, counting backwards in
time. $\hat{S}_{-n}^{(i)}$ denotes the service time of the $n$-the customer
that gets routed to server $i$, counting backwards in time.

For each $i=1,\ldots,c$, let $\{\hat{X}_{-n}^{(i)}:n\ge0\}$ with $\hat{X}%
_{0}^{(i)}=0$ be an auxiliary process
such that
\begin{equation}
\label{aux-true-eq}\hat{X}_{-n}^{(i)}:=\sum_{j=1}^{n}\left(  \hat{S}%
_{-j}^{(i)}-a\right)  =X_{-L_{n}(i)}(i).
\end{equation}
For $n\ge0$ and $1\le i\le c$, define
\begin{equation}
\label{e:hat-N-def}\hat{N}_{n}(i)=\inf\left\{  n^{\prime}\ge N_{n}%
(i):\max_{k\ge n^{\prime}}\hat{X}_{-k}^{(i)}\le\hat{X}_{-N_{n}(i)}%
^{(i)}\right\}  ,
\end{equation}
hence by definition, in (\ref{running-finite-check-X}), we have
\begin{equation}
\label{e:stopping-X}N_{n}^{X}=\max\left\{  L_{\hat{N}_{n}(1)}(1),\ldots
,L_{\hat{N}_{n}(c)}(c)\right\}  .
\end{equation}

First we develop simulation algorithms for each of the $c$ one-dimensional
auxiliary processes $\{(\hat{X}^{(i)}_{-n}:n\ge0):1\le i\le c\}$. Next we use
the common server allocation sequence $\{U_{-n}:n\ge0\}$ (sampled jointly with
the process $\{\mathbf{Y}_{-n}:n\ge0\}$ in Section \ref{section-sample-Y})
with (\ref{e:auxiliary-index}), (\ref{e:universal-index}) and
(\ref{e:service-mapping}) to find $N_{n}^{X}$ via (\ref{e:stopping-X}) for
each $n\ge0$.




\paragraph{``Milestone" construction and global maximum simulation}

\label{sec-ggm}

For each one-dimensional auxiliary process $\{\hat{X}^{(i)}_{-n}:n\ge0\}$ with
$i=1,\ldots,c$, we adopt the algorithm developed in \cite{B-W} by choosing any
$m^{\prime}>0$ and $L^{\prime}\ge1$ properly and define the sequences of
upward and downward ``milestone" events by letting $D_{0}^{(i)}=0$,
$\Gamma_{0}^{(i)}=\infty$, and for $j\ge1$,
\begin{align}
D^{(i)}_{j}  &  =\inf\{n^{(i)}\ge\Gamma^{(i)}_{j-1}I(\Gamma_{j-1}^{(i)}%
<\infty)\vee D_{j-1}^{(i)}:\hat{X}^{(i)}_{-n^{(i)}}<\hat{X}^{(i)}%
_{-D_{j-1}^{(i)}}-L^{\prime}m^{\prime}\},\label{milestone-downward-x}\\
\Gamma^{(i)}_{j}  &  =\inf\{n^{(i)}\ge D_{j}^{(i)}:\hat{X}^{(i)}_{-n^{(i)}%
}-\hat{X}^{(i)}_{-D_{j}^{(i)}}>m^{\prime}\}, \label{milestone-upward-x}%
\end{align}
with the convention that if $\Gamma_{j}^{(i)}=\infty$, then $\Gamma_{j}%
^{(i)}I(\Gamma_{j}^{(i)}<\infty)=0$ for any $j\ge0$.


For each $i=1,\ldots,c$, define
\begin{equation}
\label{global-max-lambda-x}\Lambda^{(i)}=\inf\{D^{(i)}_{k}:\Gamma^{(i)}%
_{k}=\infty,k\ge1\}.
\end{equation}
By the ``milestone" construction in (\ref{milestone-downward-x}) and
(\ref{milestone-upward-x}), for all $n\ge\Lambda^{(i)}$,
\[
\hat{X}^{(i)}_{-n}\le\hat{X}^{(i)}_{-\Lambda^{(i)}}+m^{\prime}< 0=\hat
{X}^{(i)}_{0}.
\]
Therefore we can evaluate the global maximum of the infinite-horizon process
$\{\hat{X}^{(i)}_{-n}:n\ge0\}$ in finite steps, i.e.,
\begin{equation}
\label{global-max-x}M_{0}^{(i)}:=\max_{k\ge0}\hat{X}^{(i)}_{-k}=\max_{0\le
k\le\Lambda^{(i)}}\hat{X}^{(i)}_{-k}.
\end{equation}
We summarize the simulation details in the following algorithm.

\bigskip\noindent\textbf{Algorithm} {\small GGM}: simulate global maximum of
the one-dimensional process $\{(\hat{X}_{-n}^{(i)},\hat{S}_{-n}^{(i)}%
):n\ge0\}$ jointly with the sub-path and the subsequence of ``milestone"
events.\newline

\noindent Input: $a$, $m^{\prime}$, $L^{\prime}$.

\begin{enumerate}
\item (${\mathit{I}nitialization}$) Set $n=0$, $\hat{X}^{(i)}_{0}=0$,
$\mathbf{D}^{(i)}=[0]$, $\mathbf{\Gamma}^{(i)}=[\infty]$, $L^{(i)}=0$.

\item Generate $S\sim G$. Set $n=n+1$, $\hat{X}^{(i)}_{-n}=\hat{X}%
^{(i)}_{-(n-1)}+S$ and $\hat{S}^{(i)}_{-n}=S$. \label{ggm-2}

\item If $\hat{X}_{-n}^{(i)}\ge L^{(i)}-L^{\prime}m^{\prime}$, go to Step
\ref{ggm-2}; otherwise set $\mathbf{D}^{(i)}=[\mathbf{D}^{(i)},n]$ and
$L^{(i)}=\hat{X}_{-n}^{(i)}$.

\item Call Algorithm 1 on page 10 of \cite{B-W} and obtain $(J,\omega)$.

\item If $J=1$, set $\hat{X}^{(i)}_{-(n+l)}=L^{(i)}+\omega(l)$, $\hat{S}%
^{(i)}_{-(n+l)}=\hat{X}^{(i)}_{-(n+l)}-\hat{X}^{(i)}_{-(n+l-1)}+a$ for
$l=1,\ldots,length(\omega)$. Set $n=n+length(\omega)$, $\mathbf{\Gamma}%
^{(i)}=[\mathbf{\Gamma}^{(i)},n]$ and go to Step \ref{ggm-2}.

\item If $J=0$, set $\Lambda^{(i)}=n$, $\mathbf{\Gamma}^{(i)}=[\mathbf{\Gamma
}^{(i)},\infty]$.

\item Output $\{(\hat{X}^{(i)}_{-k},\hat{S}^{(i)}_{-k}):1\le k\le\Lambda
^{(i)}\}$, $\mathbf{D}^{(i)}$, $\mathbf{\Gamma}^{(i)}$ and global maximum
$M_{0}^{(i)}=\max_{0\le k\le\Lambda^{(i)}}\hat{X}^{(i)}_{-k}$.
\end{enumerate}

\bigskip

\paragraph{Simulate $\{\mathbf{X}_{-n}:n\ge0\}$ jointly with ``milestone"
events}

\label{sec-grw} In this section, we first explain how to sample the auxiliary
one-dimensional processes $\{\hat{X}^{(i)}_{-n}:n\ge0\}$ along with the
``milestone" events defined in (\ref{milestone-downward-x}) and
(\ref{milestone-upward-x}). Next we will need the service allocation
information $\{U_{-n}:n\ge0\}$, from the simulation procedure of process
$\{\mathbf{Y}_{-n}:n\ge0\}$, to recover the multi-dimensional process of
interest $\{\mathbf{X}_{-n}:n\ge0\}$ via Equation (\ref{aux-true-eq}).

The following algorithm gives the the sampling procedure for each auxiliary
one-dimensional process $\{\hat{X}^{(i)}_{-n}:n\ge0\}$ for $i=1,\ldots,c$. The
simulation steps are the same as the procedure given in Algorithm 3 on page 16
of \cite{B-W}.

\bigskip\noindent\textbf{Algorithm} {\small GRW}: continute to sample the
process $\{(\hat{X}^{(i)}_{-k},\hat{S}^{(i)}_{-k}):0\le k\le n\}$ jointly with
the partially sampled ``milestone" event lists $\mathbf{D}^{(i)}$ and
$\mathbf{\Gamma}^{(i)}$, until a stopping criteria is met.\newline

\noindent Input: $a$, $m^{\prime}$, $L^{\prime}$, previously sampled partial
process $\{(\hat{X}^{(i)}_{-j},\hat{S}^{(i)}_{-j}):0\le j\le l\}$, partial
``milestone" sequences $\mathbf{D}^{(i)}$ and $\mathbf{\Gamma}^{(i)}$, and
stopping criteria $\mathcal{H}^{(i)}$.\newline(Note that if there is no
previously simulated random walk, we initialize $l=0$, $\mathbf{D}^{(i)}=[0]$
and $\mathbf{\Gamma}^{(i)}=[\infty]$.)

\begin{enumerate}
\item Set $n=l$. If $n=0$, call Algorithm {\small GGM} to get $\Lambda^{(i)}$,
$\{(\hat{X}^{(i)}_{-k},\hat{S}^{(i)}_{-k}):0\le k\le\Lambda^{(i)}\}$,
$\mathbf{D}^{(i)}$ and $\mathbf{\Gamma}^{(i)}$. Set $n=\Lambda^{(i)}$.

\item While the stopping criteria $\mathcal{H}^{(i)}$ is not satisfied,

\begin{enumerate}
\item Call Algorithm {\small GGM} to get $\tilde{\Lambda}^{(i)}$,
$\{(\tilde{X}^{(i)}_{-j},\tilde{S}^{(i)}_{-j}):0\le j\le\tilde{\Lambda}%
^{(i)}\}$. $\tilde{\mathbf{D}}$, $\tilde{\mathbf{\Gamma}}$ and $\tilde{M}%
_{0}^{(i)}$.

\item If $\tilde{M}_{0}^{(i)}\le m^{\prime}$, accept the proposed sequence and
concatenate it to the previous sub-path, i.e., set $\hat{X}_{-(n+j)}%
^{(i)}=\hat{X}^{(i)}_{-n}+\tilde{X}^{(i)}_{-j}$, $\hat{S}^{(i)}_{-(n+j)}%
=\tilde{S}_{-j}^{(i)}$ for $1\le j\le\tilde{\Lambda}^{(i)}$. Update the
sequences of ``milestone" events to be $\mathbf{D}^{(i)}=[\mathbf{D}%
^{(i)},n+\tilde{\mathbf{D}}^{(i)}(2:\text{end})]$, $\mathbf{\Gamma}%
^{(i)}=[\mathbf{\Gamma}^{(i)},n+\tilde{\mathbf{\Gamma}}^{(i)}(2:\text{end})]$
and set $n=n+\tilde{\Lambda}^{(i)}$.
\end{enumerate}

\item Output $\{(\hat{X}^{(i)}_{-k},\hat{S}^{(i)}_{-k}):0\le k\le n\}$ with
updated ``milestone" event sequences $\mathbf{D}^{(i)}$ and $\mathbf{\Gamma
}^{(i)}$.
\end{enumerate}

\bigskip

With the service allocation information $\{U_{-n}:n\ge0\}$, we can construct
the $c$-dimensional process $\{\mathbf{X}_{-n}:n\ge0\}$ ($\mathbf{X}%
_{0}=\mathbf{0}$) from the auxiliary processes $\{(\hat{X}^{(i)}_{-n},\hat
{S}^{(i)}_{-n}):n\ge0\}$, $i=1,\ldots,c$. For $n\ge1$,
\begin{align}
&  S_{-n} =\hat{S}^{(U_{-n})}_{\sum_{j=1}^{n}I(U_{-j}=U_{-n})}%
,\label{service-match}\\
&  X_{-n}(i) = \left\{
\begin{array}
[c]{lr}%
X_{-(n-1)}(i) & \text{if } i\neq U_{-n}\\
X_{-(n-1)}+S_{-n}-a & \text{if } i=U_{-n}%
\end{array}
.\right.  \label{x-match}%
\end{align}




By the definition of ``milestone" events in (\ref{milestone-downward-x}) and
(\ref{milestone-upward-x}), for each $n\ge0$, let
\begin{align}
d_{1}^{(i)}(n)  &  =\inf\{D_{k}^{(i)}\ge n:\hat{X}_{-D_{k}^{(i)}}^{(i)}\le
\hat{X}^{(i)}_{-n}\},\label{e:i-d1}\\
d_{2}^{(i)}(n)  &  =\inf\{D_{k}^{(i)}>d_{1}^{(i)}(n):\Gamma_{k}^{(i)}=\infty\}
. \label{e:i-d2}%
\end{align}
Since
\[
\max_{k\ge d_{2}^{(i)}(N_{n}(i))}\hat{X}^{(i)}_{-k}\le\hat{X}^{(i)}%
_{-d_{2}^{(i)}(N_{n}(i))}+m^{\prime}<\hat{X}^{(i)}_{-d_{1}^{(i)}(N_{n}(i))}%
\le\hat{X}_{-N_{n}(i)}^{(i)},
\]
we conclude that $\hat{N}_{n}(i)\le d_{2}^{(i)}(N_{n}(i))$ and hence
\[
N_{n}^{X}\le\max\{L_{d_{2}^{(1)}(N_{n}(1))}(1),\ldots,L_{d_{2}^{(c)}%
(N_{n}(c))}(c)\}.
\]

\subsubsection{Simulation algorithm for $\{\mathbf{R}_{n}^{(r)}:n\ge0\}$ and
coalescence detection}

\label{section-coalescence}

We shall combine the simulation algorithms in Section \ref{section-sample-Y}
and Section \ref{section-sample-X} for processes $\{((\hat{X}^{(i)}_{-n}%
,\hat{S}^{(i)}_{-n}):n\ge0),1\le i \le c\}$ and $\{(\mathbf{Y}_{-n}%
,U_{-n},T_{-n}):n\ge0\}$ together to exactly simulate the multi-dimensional
random walk $\{\mathbf{R}^{(r)}_{n}:n\ge0\}$ until coalescence time $N$
defined in (\ref{eq-stopping-time}). To detect the coalescence, we start from
$n=0$ to compute $d_{2}(n)$ and $d_{2}^{(i)}(N_{n}(i))$ (as defined in
(\ref{e:Y-d2}) and (\ref{e:i-d2}) respectively). If
\begin{equation}
\label{e-coalescence-Y}\max_{n\le k\le d_{2}(n)}\mathbf{Y}_{-k}=\mathbf{Y}%
_{-n},
\end{equation}
and
\begin{equation}
\label{e-coalescence-X}\max_{N_{n}(i)\le k\le d_{2}^{(i)}(N_{n}(i))}\hat
{X}^{(i)}_{-k}=\hat{X}^{(i)}_{-N_{n}(i)}%
\end{equation}
for all $i=1,\ldots,c$, we set the coalescence time $N\leftarrow n$ and stop.
Otherwise we increase $n$ by $1$ and repeat the above procedure until the
first time that (\ref{e-coalescence-Y}) and (\ref{e-coalescence-X}) are satisfied.

In the following algorithm we give the simulation procedure to detect
coalescence while sampling the time-reversed multi-dimensional process
$\{\mathbf{R}^{(r)}_{n}:n\ge0\}$.

\bigskip\noindent\textbf{Algorithm} {\small CD}: sample the coalescence time
$N$ jointly with the process $\{\mathbf{R}^{(r)}_{n}:n\ge0\}$.\newline

\noindent Input: $a$, $m$, $m^{\prime}$, $L^{\prime}$.

\begin{enumerate}
\item (\textit{Initialization}) Set $n=0$. Set $l=0$, $\mathbf{Y}%
_{0}=\mathbf{0}$, $\mathbf{D}=[0]$, $\mathbf{\Gamma}=[\infty]$. Set $l_{i}=0$,
$\hat{X}^{(i)}_{0}=0$, $\mathbf{D}^{(i)}=[0]$, $\mathbf{\Gamma}^{(i)}%
=[\infty]$ for all $i=1,\ldots,c$.

\item Call Algorithm {\small LTRW} to further sample $\{(\mathbf{Y}%
_{-j},U_{-j},T_{-j}):0\le j\le l\}$, $\mathbf{D}$ and $\mathbf{\Gamma}$ with
the stopping criteria $\mathcal{H}$ being $\sum_{j=1}^{l}I(U_{-j}=i)>l_{i}$
for all $i=1,\ldots,c$ and $\mathbf{Y}_{-\mathbf{D}(end-1)}\le\mathbf{Y}_{-n}$.

\item For each $i=1,\ldots,c$,

\begin{enumerate}
\item Set $n_{i}=\sum_{j=1}^{n}I(U_{-j}=i)$.

\item Call Algorithm {\small GRW} to further sample $\{(\hat{X}^{(i)}%
_{-k},\hat{S}^{(i)}_{-k}):0\le k\le l_{i}\}$, $\mathbf{D}^{(i)}$ and
$\mathbf{\Gamma}^{(i)}$ with the stopping criteria $\mathcal{H}^{(i)}$ being
$\sum_{j=1}^{l}I(U_{-j}=i)\le l_{i}$ and $\hat{X}^{(i)}_{-\mathbf{D}%
^{(i)}(end-1)}\le\hat{X}^{(i)}_{-n_{i}}$.
\end{enumerate}

\item If $\max_{n\le k\le\mathbf{D}(end)}\mathbf{Y}_{-k}\le\mathbf{Y}_{-n}$
and $\max_{n_{i}\le k\le\mathbf{D}^{(i)}(end)}\hat{X}^{(i)}_{-k}\le\hat
{X}^{(i)}_{-n_{i}}$ for all $i=1,\ldots,c$, go to next step. Otherwise set
$n=n+1$ and go to Step 2.

\item For $1\le k\le n$, recover $S_{-k}$ and $\mathbf{X}_{-k}$ from the
auxiliary processes via Equations (\ref{service-match}) and (\ref{x-match}).

\item Output coalescence time $N=n$, the sequence $\{(U_{-k},T_{-k}%
,S_{-k}):0\le k\le n\}$ and process $\{\mathbf{R}^{(r)}_{k}:0\le k\le n\}$.
\end{enumerate}

\bigskip


\subsection{Proofs}

\label{proofs}

\begin{proof}
[Proof of Proposition \ref{alg1-prop}] Firstly, $E(N)<\infty$ holds true under
assumptions $\rho<c$ and $P(T>S)>0$ (proved in \cite{KS-1988}). Next we shall
prove the computational effort $\tau$ has finite expectation as well.

For $n\ge0$, we have $N_{n}^{X}$, $N_{n}^{Y}$ and $N_{n}$ defined in Equations
(\ref{running-finite-check-X} - \ref{running-finite-check}) such that
\[
\max_{k\ge N_{n}}\mathbf{R}^{(r)}_{k}\le\max_{k\ge N_{n}}\mathbf{X}_{-k}%
+\max_{k\ge N_{n}}\mathbf{Y}_{-k}\le\mathbf{X}_{n}+\mathbf{Y}_{n}%
=\mathbf{R}_{n}^{(r)}.
\]
Therefore, in order to evaluate the running-time maximum over the infinite
horizon $\max_{k\ge n}\mathbf{R}_{k}^{(r)}$, it only requires sampling from
$n$ to $N_{n}$ backwards in time, i.e.,
\[
\max_{k\ge n}\mathbf{R}_{k}^{(r)}=\max\{\max_{n\le k\le N_{n}}\mathbf{R}%
_{k}^{(r)},\max_{k\ge N_{n}}\mathbf{R}_{k}^{(r)}\}=\max_{n\ge k\le N_{n}%
}\mathbf{R}_{k}^{(r)}.
\]
An easy upper bound for $\tau$ is given by $\tilde{\tau}=\sum_{n=0}^{N}N_{n}$.
By Wald's identity, it suffices to show that $E(N_{n})<\infty$ for any $n\ge0$.

By the ``milestone" events construction for multi-dimensional process
$\{\mathbf{Y}_{-n}:n\ge0\}$ in (\ref{e:upward-milestone}),
(\ref{e:downward-milestone}) and because $d_{2}(n)$ is an upper bound of
$N_{n}^{Y}$, $E(N_{n}^{Y})\le E(d_{2}(n))<\infty$ follows directly from
elementary properties of compound geometric random variables (see Theorem 1 of
\cite{B-C}).

For the other process $\{\mathbf{X}_{-n}:n\ge0\}$, we simulate each of its $c$
entries separately, i.e., $\{\{\hat{X}^{(i)}_{-n}:n\ge0\}:1\le i\le c\}$ in
Section \ref{section-sample-X}. Equation (\ref{e:stopping-X}) gives
\[
N_{n}^{X}=\max\{L_{\hat{N}_{n}(1)}(1),\ldots,L_{\hat{N}_{n}(c)}(c)\}\le
\sum_{i=1}^{c}L_{\hat{N}_{n}(i)}(i)
\]
where $\hat{N}_{n}(i)$ is defined in (\ref{e:hat-N-def}). By Theorem 2.2 of
\cite{B-W}, $E(\hat{N}_{n}(i))<\infty$. Because
\[
L_{\hat{N}_{n}(i)}(i)=\inf\{k\ge0:\sum_{j=1}^{k}I(U_{-j}=i)=\hat{N}%
_{n}(i)\}\sim NegBinomial\left(  \hat{N}_{n}(i);1-\frac{1}{c}\right)  +\hat
{N}_{n}(i),
\]
hence
\[
E(L_{\hat{N}_{n}(i)}(i))=(c-1)E(\hat{N}_{n}(i))+E(\hat{N}_{n}(i))=cE(\hat
{N}_{n}(i))<\infty,
\]
and
\[
E(N_{n}^{X})\le\sum_{i=1}^{c}E(L_{\hat{N}_{n}(i)}(i))<\infty.
\]

Therefore
\[
E(N_{n})\le E(N_{n}^{X})+E(N_{n}^{Y})<\infty.
\]

\end{proof}

\begin{proof}
[Proof of Proposition \ref{alg2-prop}] By Wald's identity, it suffices to show
that $E(\kappa^{*}_{+})<\infty$ because $E(T)<\infty$. Next we only provide a
proof outline here since it follows the same argument as in the proof of
Proposition 3 in \cite{B-D-P}.

Firstly, we construct a sequence of events $\{\Omega_{k}:k\ge1\}$ which leads
to the occurrence of $\kappa^{*}_{+}$. Secondly, we split the process
$\{\mathbf{W}_{0}^{u}(t_{n}):n\ge0\}$ into cycles with bounded expected cycle
length. We also ensure the probability that the event happens during each
cycle is bounded from below by a constant, which allows us to bound the number
of cycles we need to check before finding $\kappa^{*}_{+}$ by a geometric
random variable. Finally we could establish an upper bound for $E(\kappa
^{*}_{+})$ by applying Wald's identity again.
\end{proof}

\textbf{Acknowledgement:} Support from NSF through grant CMMI-1538217 is
gratefully acknowledged.

\end{document}